\documentclass[12pt,a4paper,twoside]{article}
\usepackage[a4paper,left=30mm,right=30mm,top=26mm,bottom=25mm]{geometry}

\usepackage[latin1]{inputenc}
\usepackage{amsmath}
\usepackage{amsthm}
\usepackage{amsfonts}
\usepackage{amssymb}
\usepackage{stmaryrd}
\usepackage{graphicx}
\usepackage{epstopdf}
\usepackage{mathtools}
\usepackage{color}
\usepackage{bbm}
\usepackage{MnSymbol}
\usepackage{caption,subcaption}
\usepackage{float}
\usepackage{wrapfig}
\usepackage{caption}
\usepackage{cite}
\usepackage{cleveref}
\usepackage{cases}
\usepackage{enumitem}
\usepackage{tikz}
\usepackage{ifthen}
\DeclareMathSizes{10}{10}{10}{10}
\graphicspath{ {Figures/} }

\def \calT  {\mathcal{T}}

\def \e  {\varepsilon}
\def \eps {\varepsilon}
\def \fhi {\varphi}
\def \vr  {\varrho}
\def \Om {\Omega}

\def \t  {\tau}

\def \del {\partial}
\def \p  {\partial}
\def \R  {\mathbb{R}}
\def \N  {\mathbb{N}}
\def \sgn {{\rm sign}}
\def \si {\smallint}

\def \cut {\mathbbm{1}}

\newtheorem{theorem}{Theorem}[section]
\newtheorem{definition}[theorem]{Definition}
\newtheorem{lemma}[theorem]{Lemma}
\newtheorem{proposition}[theorem]{Proposition}

\newtheorem{assumption}[theorem]{Assumption}

\newtheorem{remark}[theorem]{Remark}

\numberwithin{equation}{section}

\begin{document}

\thispagestyle{empty}
\begin{center}
  ~\vskip7mm {\Large\bf Travelling wave solutions for\\[3mm]
    gravity fingering in porous media flows
  }\\[8mm]

  {\large K. Mitra, A. R\"atz, and B. Schweizer}\\[7mm]

  November 10, 2020\\[4mm]
\end{center}

\pagestyle{myheadings}
\markboth{Travelling wave solutions for gravity
  fingering}{K. Mitra, A. R\"atz, and B. Schweizer}

\begin{center}
   \vskip2mm
   \begin{minipage}[c]{0.8\textwidth}
     {\bf Abstract:} We study an imbibition problem for porous media.
     When a wetted layer is above a dry medium, gravity leads to the
     propagation of the water downwards into the medium. In
     experiments, the occurence of fingers was observed, a phenomenon
     that can be described with models that include hysteresis. In the
     present paper we describe a single finger in a moving frame and
     set up a free boundary problem to describe the shape and the
     motion of one finger that propagates with a constant speed. We
     show the existence of solutions to the travelling wave problem
     and investigate the system numerically.
     \\[-1mm]

    {\bf MSC:} 76S05, 35C07, 47J40\\[-1mm]

    
    {\bf Keywords:} porous media, travelling waves, hysteresis
   \end{minipage}\\[2mm]
\end{center}

\section{Introduction}

Standard models for flow in unsaturated porous media fail in the
description of a fundamental process, namely the imbibition into a dry
medium with gravity as the driving force. While standard Richards models
predict the formation of uniform imbibition fronts, the experimentally
observed fingers \cite{glass1989mechanism,selker1992fingered} can only
be described with a model that incorporates hysteresis.

Models for incompressible unsaturated porous media flow typically use
the water pressure $p$ and the water saturation $s$ as primary
variables. The Darcy law for the velocity together with the mass
balance equation leads to
\begin{subequations}\label{eq:ParabolicGDE}
  \begin{equation}
    \del_t s = \nabla \cdot (k(s)[\nabla p + g e_z])\,,\label{eq:Richards}
  \end{equation}
  we refer to \cite{Richards_hys, Bear1979, helmig1997multiphase,
    schweizer2017hysteresis} for the modelling.  In the Richards
  equation \eqref {eq:Richards}, the function $k : [0,1] \to \R$ is
  the permeability function which has to be determined from
  experiments, $g$ is the gravitational acceleration, $e_z$ is the
  normal vector pointing upwards.  It is always assumed that $s$ takes
  only values in $[0,1]$.

  Equation \eqref{eq:Richards} must be accompanied by a relation
  between saturation $s$ and pressure $p$.  Models without hysteresis
  demand either the algebraic relation $p = p_c(s)$ for some given
  function $p_c:[0,1] \to \bar\R$, or they include the
  ``$\tau$-correction'' and demand, for some physical parameter
  $\tau>0$, known as the dynamic capillary number, that
  $p=p_c(s)+ \t\del_t s$; this latter model takes inertia in the
  material law into account, see \cite{hassanizadeh1993thermodynamic}.
  If, additionally, hysteresis in an imbibition process shall be
  modelled, a possible simple law is
  \begin{equation}
    \del_t s=\frac{1}{\t}[p-p_c(s)]_+\,,\label{eq:closure}
  \end{equation}
\end{subequations}
where $[\cdot]_+:=\max\{0,\cdot\}$ denotes the positive part. Our aim
is a travelling wave analysis of equation \eqref {eq:ParabolicGDE}.
We recall that $p_c:(0,1)\to \R$ is a given imbibition capillary
pressure function and $\t>0$ is a given constant.

Regarding the modelling we note that, if both imbibition and drainage
should be modelled, one replaces \eqref{eq:closure} by the model of
\cite{BeliaevSchotting},
\begin{align}
  \del_t s = \frac{1}{\t}[p-p_c(s)]_+ + \frac{1}{\t}[p-p_d(s)]_-\,.
  \label{eq:playtype}
\end{align}
Here, $p_d: (0,1)\to \R$ is a drainage capillary pressure function
with $p_d(s) \le p_c(s)$ for all $s\in (0,1)$, and
$[\cdot]_-:=\min\{0,\cdot\}$ is the negative part function.  Equation
\eqref {eq:playtype} is a hysteresis model since, pointwise in space
and time, all pressure values in the closed interval
$[p_d(s), p_c(s)]$ are permitted for a fixed saturation $s$.  The
play-type hysteresis model with dynamic capillary pressure was
analyzed in \cite{BeliaevSchotting, VANDUIJN2018232, mitra2018wetting,
  mitra2019fronts, LamaczRaetzSchweizer2011, SCHWEIZER20125594,
  ratz2014hysteresis, schweizer2012instability, KMitraEUX_2020}.  Since we are
interested in an infiltration problem with $\del_t s \geq 0$, we
restrict ourselves to the case $p_d(s)= -\infty$ as in \cite
{el2018traveling}, i.e., we study \eqref{eq:closure} instead of
\eqref{eq:playtype}.

\begin{figure}[H]
  \centering
  \includegraphics[width=.29\textwidth]{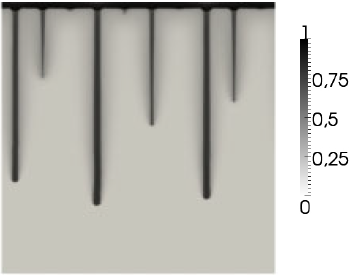}
  \quad
  \includegraphics[width= 0.22 \textwidth]{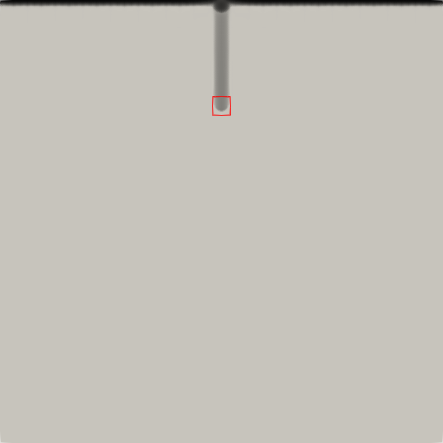}
  \includegraphics[width= 0.07 \textwidth]{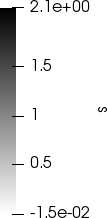}
  \quad
  \includegraphics[width= 0.22 \textwidth]{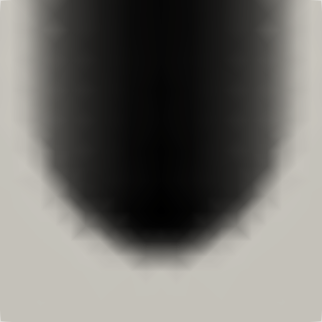}
  \includegraphics[width= 0.07 \textwidth]{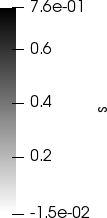}
  \caption{Motivation for this contribution. Left: A snapshot of a
    solution to the time dependent system \eqref
    {eq:ParabolicGDE}. Fingers are clearly visible; the solution is
    comparible to experimental observations
    \cite{ratz2014hysteresis}. Middle: With another choice of boundary
    values and $p_c$, a single finger is generated.  A small squared
    region of size $2 \times 2$ around the finger-tip is marked
    \cite{LamaczRaetzSchweizer2011}. Right: Enlargement of the marked
    region. We see the typical shape of the single finger in
    time-dependent calculations.  The aim of this contribution is to
    analyze the travelling wave equations corresponding to
    \eqref{eq:ParabolicGDE} in order to obtain the shape of the single
    finger without a time-dependent
    calculation.}\label{Fngr1_fig:Ratz}
\end{figure}

Numerical results for the time dependent system \eqref
{eq:ParabolicGDE} are shown in Figure \ref {Fngr1_fig:Ratz},
originally published in
\cite{ratz2014hysteresis,LamaczRaetzSchweizer2011}.  The figure
illustrates a gravity driven imbibition process into an originally dry
medium. Several fingers evolve in the process. It is observed that
each finger travels approximately with constant speed. This has also
been verified experimentally \cite{selker1992fingered}. The present
work aims at the description of a single finger in a co-moving frame
of coordinates.

\paragraph{Travelling wave ansatz, domains and boundary conditions.}
Since we are interested in imbibition fronts in columns of porous
media, we choose a cylindrical spatial domain
$\Om_\infty$. Restricting to two dimensions for convenience and
denoting the width of the cylinder by $L>0$, we consider
$\Om_\infty := (0,L)\times \R \subset \R^2$. Points in $\R^2$ are
denoted as $x = (y,z)$.  We seek time-dependent solutions to
\eqref{eq:ParabolicGDE} that move with a constant speed $c>0$ in
negative $z$-direction, i.e., downwards. This motivates the travelling
wave coordinates
\begin{equation}
  \tilde{z} = z+ ct\,,\quad p(y,z,t) = p(y,\tilde{z})\,, \quad
  s(y,z,t) = s(y,\tilde{z})\,.
\end{equation}
In the following, we omit the tilde symbol and write $z$ instead of
$\tilde{z}$.  The new coordinates transform system
\eqref{eq:ParabolicGDE} into
\begin{subequations}\label{eq:GDETW}
  \begin{align}
    &c\del_z s=\nabla \cdot (k(s)[\nabla p + g e_z])\,,\label{eq:Richards2}\\
    &c\t \del_z s= [p - p_c(s)]_+\,.\label{eq:closure2}
  \end{align}
\end{subequations}

Even though the physical interpretation of a travelling wave solution
requires the study of domains $\Om_\infty$ that extend to
$z\to \pm\infty$, we choose here to study problem \eqref {eq:GDETW} on
the semi-infinite domain
\begin{equation*}
  \Om := (0,L)\times \R_+\quad \text{ with bottom }\quad
  \Sigma:= (0,L)\times \{0\} = \{(y,0):0<y<L\}\,.
\end{equation*}
Truncations of the domain are necessary for numerical calculations and
facilitate the analysis. The problem is translation invariant; one
should consider the bottom $\Sigma = \{z=0\}$ as being far below the
finger.

The boundary data are given by a prescribed saturation $s_0 > 0$ and a
prescribed pressure $p_0$ at the bottom $\Sigma$ of the domain, and by
a prescribed total influx $F_\infty$ on the top of the domain.  More
precisely, we assume that we are given $s_0 : [0,L] \to [0,1]$,
$p_0 : [0,L] \to \R$, and $F_\infty\in \R_+ = (0,\infty)$, and impose
the boundary conditions
\begin{subequations}\label{eq:ParabolicBC}
  \begin{align}\label{eq:ParabolicBC-top}
    &\int_0^L k(s(y,z))[\del_z p(y,z) + g]\, dy \to F_\infty
    \quad \text{ as } z\to +\infty\,,\\
    &s = s_0\qquad \text{ at } z = 0\,,\\
    &p = p_0\qquad \text{ at } z = 0\,.
  \end{align}
\end{subequations}
If the initial saturation of the medium is given by a number
$s_*\in (0,1)$, a natural choice for the boundary data is
$s_0 \equiv s_*$ and $p_0 \equiv p_c(s_*)$.  Along the lateral
boundaries of $\Om$ we impose homogeneous Neumann conditions (no
flux).

\paragraph{Main results.}

We perform an analysis of the travelling wave problem \eqref
{eq:GDETW}--\eqref {eq:ParabolicBC} on $\Om$. For the most part of
this article, we prescribe the relaxation parameter $\tau$, the frame
speed $c$, and the boundary data $s_0$, $p_0$, and $F_\infty$. Only in
our last result, Theorem \ref {thm:c-F-infty}, we choose $c$ in
dependence of the other parameters in order to satisfy a physically
adequate flux condition on the lower boundary.

The first part of our results concerns the system \eqref
{eq:GDETW}--\eqref {eq:ParabolicBC} on the bounded truncated domain
$\Omega^H = (0,L) \times (0,H)$. We choose boundary conditions on the
upper boundary appropriately and show that the system has a
solution. The solution can be found with a variational principle, the
analysis is given in Section \ref {sec.truncated-domain}.

The numerical part of this paper deals with this truncated
problem. One result is the calculation of a finger solution, see
Figure \ref{Fig:introFingers}.  The numerical method and the results
are described in Section \ref {sec.numerics}.

The limit $H\to \infty$ for the solutions on the bounded domain is
studied in Section \ref {sec.unbounded-domain}. We find that every
sequence of solutions $(s_H, p_H)$ to truncated domain problems
possesses a subsequence and a limit $(s,p)$ which is a solution of the
original problem \eqref {eq:GDETW}. The limit process shows an
interesting dichotomy: In one case, the flux boundary condition for
$z\to \infty$ as in \eqref {eq:ParabolicBC} remains satisfied (``large
solution'').  In the other case (``small solution''), only a
corresponding inquality is satisfied.

The two cases are analyzed further. We find that ``large solutions''
are of the type that we would like to see in the fingering process:
they possess a free boundary, the pressure $p$ tends to $-\infty$ as
$z\to \infty$, and the solution is ``large'' in the sense that the
saturation exceed a certain threshold.  In the second case, the
properties are reverted: The solution has a bounded pressure and it is
``small'' in the same sense as the solution was ``large'' in the other
case. Interestingly, both types of solutions are found numerically,
see Section \ref {sec.numerics}.

{\bf Free boundary problem.} Let us emphasize that we treat a free
boundary problem. By \eqref {eq:closure2}, one has to distinguish
between the subdomain $\{x\in\Om\, |\, \del_z s(x) > 0\}$ (expected to
be in the bottom) and the subdomain
$\{x\in\Om\, |\, \del_z s(x) = 0\}$ (expected in the top part). In
physical terms, this means that an imbibition process occurs near and
below the finger-tip, whereas, in the region around the developed
finger, the saturation does not change any more. With reference to the
hysteresis relation, we note that the $z$-independent saturation
implies that the pressure can take arbitrary values (below
$\min p_c(s)$). Therefore, the pressure profile does not have to
reflect the saturation profile and the fingers can remain stable in
their upper part; no blurring by pressure differences occurs.

\begin{figure}[ht]
  \centering
  \includegraphics[width=0.39\textwidth]{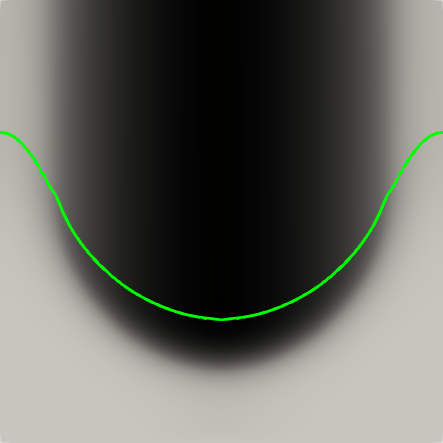}
  \includegraphics[width=0.09\textwidth]{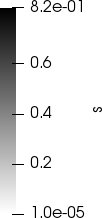}
  \includegraphics[width=0.39\textwidth]{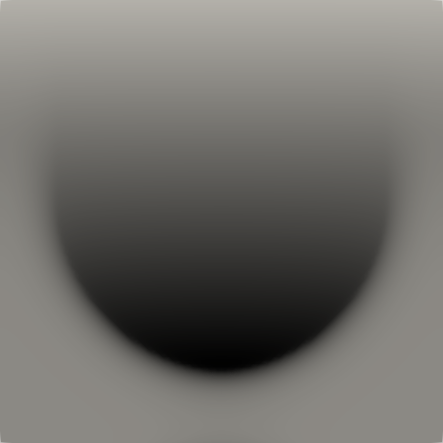}
  \includegraphics[width=0.09\textwidth]{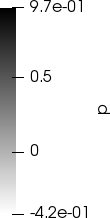}
  \caption{A numerical solution of the free boundary travelling wave
    problem. The gray scale indicates the values of the saturation $s$
    (left) and the pressure $p$ (right).  The level line
    $\Gamma = \{x\,|\, p = p_c(s)\}$ is marked in the left image. The
    line $\Gamma$ shows the free boundary: Below the line, the
    saturation is increasing, above the line, the saturation remains
    constant (increasing in vertical direction and, hence, increasing
    in time when interpreted as a time dependent solution).
  }\label{Fig:introFingers}
\end{figure}

With Theorem \ref {thm:c-F-infty} we provide the result that, for
every $F_\infty$ within appropriate bounds, there exists a wave speed
$c$ such that a physical flux condition at the lower boundary is
satisfied.

\paragraph{Literature.}

The classical porous media equation is obtained by setting $\tau = 0$
and by replacing \eqref {eq:closure} by the algebraic law
$p=p_c(s)$. This classical equation is interesting when the
permeability coefficient is degenerate $k(0) = 0$. For existence and
uniqueness results in this classical case we refer to \cite
{AltLuckhaus, OttoL1}. The hysteresis model \eqref {eq:closure} was
introduced in \cite {hassanizadeh1993thermodynamic,
  BeliaevHassanizadeh, BeliaevSchotting}.  It combines dynamic effects
($\tau>0$) with a play-type hysteresis relation; the latter allows for
an interval of pressure values $p$ for a fixed saturation $s$. For a
review of the modelling, we refer to \cite {schweizer2017hysteresis}.

For the model \eqref {eq:ParabolicGDE}, well-posedness results have
been obtained in one space dimension in \cite {BeliaevSchotting}, and
in higher dimension in \cite {LamaczRaetzSchweizer2011,
  ratz2014hysteresis}. Existence of solutions for an extension of the
play-type model was shown in \cite{KMitraEUX_2020}. In \cite
{schweizer2012instability}, it was shown that the model does not
define an $L^1$-contraction; in this sense, it can explain the
fingering effect. The fingers were found numerically for unsaturated
media in \cite {LamaczRaetzSchweizer2011}, for the two-phase flow in
\cite {KochRaetzSchweizer2013}. Fingers were also observed numerically
in \cite{beljadid2020continuum,cueto2008nonlocal}, where a free-energy
based approach is used for modelling the capillary pressure. For a
result with a degenerate $p_c$-curve, see \cite {SCHWEIZER20125594}.
A uniqueness result was derived in \cite {CaoPop2015}.

Travelling waves for the model have been analyzed in
\cite{VANDUIJN2018232, mitra2018wetting, mitra2019fronts}.  An
analysis for pure imbibition ($\del_t s \geq 0$ allows to set
$p_d(s)= -\infty$) was previously performed for one space dimension in
\cite {el2018traveling}. The present work extends the results to two
space dimensions. Let us note that the methods are independent of the
dimension and that, up to notation, the results remain valid, e.g., in
three space dimensions. The dimension enters only in Sobolev
embeddings that are used for regularity statements in the appendix.

\section{Preliminaries}\label{sec.Preliminaries}

The coefficient functions $k$ and $p_c$ are fixed throughout this
work. We make assumptions that are quite common and consistent with
experiments, see \cite{helmig1997multiphase}.  For an illustration see
\Cref{Fingr1_fig:PcW}.

\begin{assumption} The functions $k:[0,1] \to [0,\infty)$ and
  $p_c : (0,1) \to \R$ satisfy:
  \begin{description}
  \item[(Ass-pc)] The function $p_c$ is differentiable and for some
    $\rho>0$ holds $p_c'\geq \rho$ on $(0,1)$. Upon normalization of
    the pressure, we can set $p_c(s_*) = 0$ for a given saturation
    value $s_*\in \R$. We assume $p_c(s)\to -\infty$ as $s\searrow 0$
    and $p_c(s)\to \infty$ as $s\nearrow 1$.
    \label{FingrAss1}
  \item[(Ass-k)] The function $k$ is differentiable,
    $k|_{(0,1)}\in C^2$, and $k'(.), k''(.) > 0$ on $(0,1)$.
    \label{FingrAss2}
  \end{description}
\end{assumption}

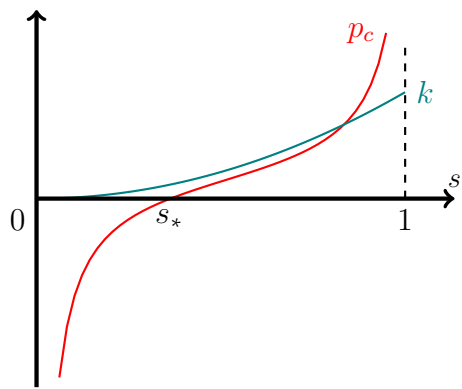
\begin{figure}[h!]
\centering
  \begin{tikzpicture}
[xscale=5,yscale=.5,domain=0:1,samples=100]

\draw[thick, red]
(0.06,-4.7422)--(0.08,-3.3947)--(0.1,-2.5777)--(0.12,-2.0257)--(0.14,-1.6251)--(0.16,-1.319)--(0.18,-1.0757)--(0.2,-0.87638)--(0.22,-0.70879)--(0.24,-0.56489)--(0.26,-0.43906)--(0.28,-0.32727)--(0.3,-0.22654)--(0.32,-0.13462)--(0.34,-0.049755)--(0.36,0.029436)--(0.38,0.10407)--(0.4,0.17508)--(0.42,0.24324)--(0.44,0.30924)--(0.46,0.37367)--(0.48,0.43709)--(0.5,0.5)--(0.52,0.56291)--(0.54,0.62633)--(0.56,0.69076)--(0.58,0.75676)--(0.6,0.82492)--(0.62,0.89593)--(0.64,0.97056)--(0.66,1.0498)--(0.68,1.1346)--(0.7,1.2265)--(0.72,1.3273)--(0.74,1.4391)--(0.76,1.5649)--(0.78,1.7088)--(0.8,1.8764)--(0.82,2.0757)--(0.84,2.319)--(0.86,2.6251)--(0.88,3.0257)--(0.9,3.5777)--(0.92,4.3947) node[left,scale=1] {$p_c$};

\draw[thick, teal, domain=0:.97] plot(\x,{3*\x*\x}) node[right,scale=1] {$k$};

\node[scale=1, below]  at (0.35,0) {$s_*$};

\node[scale=1, below]  at (.97,0) {$1$};
\node[scale=1, below left]  at (0,0) {$0$};
\draw[dashed,thick] (.97,0)--(.97,4);

\draw[->,ultra thick] (0,-5)--(0,5);
\draw[->,ultra thick] (0,0)--(1.1,0) node[scale=.9, above] {$s$};
\end{tikzpicture}
  \caption{Typical functions $p_c$ and $k$.}\label{Fingr1_fig:PcW}
\end{figure}

\paragraph{The free boundary description.}

What qualitative behavior can we expect for solutions of the
travelling wave problem \eqref {eq:GDETW}--\eqref {eq:ParabolicBC}?
We expect that the pressure stabilizes, as $z\to +\infty$, to an
affine function with $\nabla p \approx - g_F e_z$. If $s$ (and hence
$k(s)$) does not depend on $z$, then both sides of \eqref
{eq:Richards2} can vanish. This is what we expect for solutions in the
upper part of the domain. We will be interested in solutions $(p,s)$
that satisfy, for some $h \in \R_+$,
\begin{equation}\label{eq:const-above-h}
  \del_z s=0 \text{ and } p\leq p_c(s) \text{ for all } (y,z) \text{ with }
  y\in (0,L) \text{ and } z> h\,.
\end{equation}
For such a solution we can define a function
$\Psi : [0,L]\to [0,\infty)$ as
\begin{align}
  \Psi(y) := \inf \left\{ z_0>0 \middle|\; \del_z s(y,z)=0 
  \text{ for all } z\geq z_0 \right\}.
\label{eq:PSiDef}
\end{align}
The graph of $\Psi$ is a part of the free-boundary,
$\{ (y,\Psi(y))|\; y\in (0,L)\} \subset \Gamma$.  For the rest of the
paper, we define the function $s^*:[0,L]\to [0,1]$ as
\begin{equation}
  s^*(y):= \lim\limits_{z\to \infty} s(y,z).\label{eq:DefSstar}
\end{equation}
By positivity $\del_z s(y,z) \ge 0$ and boundedness of $s$, the
function $s^*$ is well-defined for solutions $(s,p)$ of \eqref
{eq:GDETW}. When a solution satisfies \eqref {eq:const-above-h}, there
holds $s(y,z) = s^*(y)$ for all $z>h$.

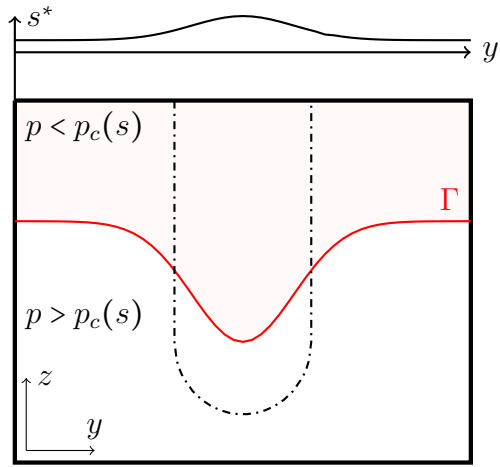
\begin{figure}[h!]
  \centering
  \begin{tikzpicture}
[xscale=3,yscale=3.2,domain=0:1,samples=100]

\draw[thick, fill, pink!10]
(-1,.5)--(-1,-2.27e-05)--(-0.96,-4.9718e-05)--(-0.92,-0.00010546)--(-0.88,-0.00021667)--(-0.84,-0.00043111)--(-0.8,-0.00083078)--(-0.76,-0.0015505)--(-0.72,-0.0028028)--(-0.68,-0.0049067)--(-0.64,-0.0083195)--(-0.6,-0.013662)--(-0.56,-0.021728)--(-0.52,-0.033469)--(-0.48,-0.049929)--(-0.44,-0.07214)--(-0.4,-0.10095)--(-0.36,-0.13681)--(-0.32,-0.17958)--(-0.28,-0.22829)--(-0.24,-0.28107)--(-0.2,-0.33516)--(-0.16,-0.38707)--(-0.12,-0.43294)--(-0.08,-0.469)--(-0.04,-0.49206)--(0,-0.5)--(0.04,-0.49206)--(0.08,-0.469)--(0.12,-0.43294)--(0.16,-0.38707)--(0.2,-0.33516)--(0.24,-0.28107)--(0.28,-0.22829)--(0.32,-0.17958)--(0.36,-0.13681)--(0.4,-0.10095)--(0.44,-0.07214)--(0.48,-0.049929)--(0.52,-0.033469)--(0.56,-0.021728)--(0.6,-0.013662)--(0.64,-0.0083195)--(0.68,-0.0049067)--(0.72,-0.0028028)--(0.76,-0.0015505)--(0.8,-0.00083078)--(0.84,-0.00043111)--(0.88,-0.00021667)--(0.92,-0.00010546)--(0.96,-4.9718e-05)--(1,-2.27e-05)--(1,.5);

\draw[ultra thick]
(-1,.5)--(1,.5)--(1,-1)--(-1,-1)--(-1,.5);

\draw[thick, red]
(-1,-2.27e-05)--(-0.96,-4.9718e-05)--(-0.92,-0.00010546)--(-0.88,-0.00021667)--(-0.84,-0.00043111)--(-0.8,-0.00083078)--(-0.76,-0.0015505)--(-0.72,-0.0028028)--(-0.68,-0.0049067)--(-0.64,-0.0083195)--(-0.6,-0.013662)--(-0.56,-0.021728)--(-0.52,-0.033469)--(-0.48,-0.049929)--(-0.44,-0.07214)--(-0.4,-0.10095)--(-0.36,-0.13681)--(-0.32,-0.17958)--(-0.28,-0.22829)--(-0.24,-0.28107)--(-0.2,-0.33516)--(-0.16,-0.38707)--(-0.12,-0.43294)--(-0.08,-0.469)--(-0.04,-0.49206)--(0,-0.5)--(0.04,-0.49206)--(0.08,-0.469)--(0.12,-0.43294)--(0.16,-0.38707)--(0.2,-0.33516)--(0.24,-0.28107)--(0.28,-0.22829)--(0.32,-0.17958)--(0.36,-0.13681)--(0.4,-0.10095)--(0.44,-0.07214)--(0.48,-0.049929)--(0.52,-0.033469)--(0.56,-0.021728)--(0.6,-0.013662)--(0.64,-0.0083195)--(0.68,-0.0049067)--(0.72,-0.0028028)--(0.76,-0.0015505)--(0.8,-0.00083078)--(0.84,-0.00043111)--(0.88,-0.00021667)--(0.92,-0.00010546)--(0.96,-4.9718e-05)--(1,-2.27e-05)
 node[above left,scale=1] {$\Gamma$};

\draw[thick, dashdotted] (-.3,.5)--(-.3,-.5)--(-0.29,-0.57681)--(-0.28,-0.6077)--(-0.27,-0.63077)--(-0.26,-0.64967)--(-0.22,-0.70396)--(-0.18,-0.74)--(-0.14,-0.76533)--(-0.1,-0.78284)--(-0.06,-0.79394)--(-0.02,-0.79933)--(0.02,-0.79933)--(0.06,-0.79394)--(0.1,-0.78284)--(0.14,-0.76533)--(0.18,-0.74)--(0.22,-0.70396)--(0.26,-0.64967)--(0.27,-0.63077)--(0.28,-0.6077)--(0.29,-0.57681)--(0.3,-0.5)--(0.3,0.5);

\node[scale=1, below right]  at (-1,.5) {$p<p_c(s)$};
\node[scale=1, above right]  at (-1,-.5) {$p>p_c(s)$};


\draw[->] (-.95,-.95)--(-.95,-.65) node[scale=1, right] {$z$};
\draw[->] (-.95,-.95)--(-.65,-.95) node[scale=1, above] {$y$};

\draw[thick, shift={(0.0,0.1)}] (-1,0.65)--(-0.96,0.65001)--(-0.92,0.65002)--(-0.88,0.65004)--(-0.84,0.65009)--(-0.8,0.65017)--(-0.76,0.65031)--(-0.72,0.65056)--(-0.68,0.65098)--(-0.64,0.65166)--(-0.6,0.65273)--(-0.56,0.65435)--(-0.52,0.65669)--(-0.48,0.65999)--(-0.44,0.66443)--(-0.4,0.67019)--(-0.36,0.67736)--(-0.32,0.68592)--(-0.28,0.69566)--(-0.24,0.70621)--(-0.2,0.71703)--(-0.16,0.72741)--(-0.12,0.73659)--(-0.08,0.7438)--(-0.04,0.74841)--(0,0.75)--(0.04,0.74841)--(0.08,0.7438)--(0.12,0.73659)--(0.16,0.72741)--(0.2,0.71703)--(0.24,0.70621)--(0.28,0.69566)--(0.32,0.68592)--(0.36,0.6736)--(0.4,0.67019)--(0.44,0.66443)--(0.48,0.65999)--(0.52,0.65669)--(0.56,0.65435)--(0.6,0.65273)--(0.64,0.65166)--(0.68,0.65098)--(0.72,0.65056)--(0.76,0.65031)--(0.8,0.65017)--(0.84,0.65009)--(0.88,0.65004)--(0.92,0.65002)--(0.96,0.65001)--(1,0.65);

\draw[->, thick] (-1,.7)--(1,.7) node[right,scale=1] {$y$};
\draw[->,thick] (-1,.5)--(-1,.85) node[right, scale=1] {$s^*$};
\end{tikzpicture}
  \caption{When interpreted as a solution of the time-dependent
    problem, the finger moves with a constant speed downwards. The
    dashed line represents the boundary of the finger; one may think
    of an isoline of the saturation.  The graph at the top part of the
    Figure indicates a profile of the limiting saturation $s^*$ as
    defined in \eqref{eq:DefSstar}.}\label{fig:freeboundary00}
\end{figure}

We refer to Figure \ref {fig:freeboundary00} for an illustration. It
is important not to confuse the free boundary $\Gamma$ with the shape
of the finger (the region of high saturation). We emphasize that the
saturation profile remains unchanged (independent of $z$) above
$\Gamma$; in particular, the finger extends to $z \to +\infty$.

\paragraph{Relations in the travelling wave formulation.}
A fundamental problem in travelling wave analysis is the determination
of free parameters, in our case the wave speed $c$. The other
parameters are fixed: $\tau, g>0$ are physical constants, $L>0$ a
geometrical constant, and the boundary conditions fix $F_\infty>0$ and
$s_*>0$. In the travelling wave formulation, $c\ge 0$ is a further
unknown of the system. Nevertheless, for the most part of our
analysis, we fix boundary values $s_0$ and $p_0$ and treat the problem
with prescribed $c$. Only in our final result we determine $c$ from an
additional boundary condition for $z\to -\infty$.

Let us collect some properties of the real parameters.

\begin{lemma}[Wave speed and limiting pressure in the doubly infinite
  domain]\label{lem:WaveSpeed-Ominfty} Let $(s,p)\in
  C^1(\Om_\infty)\times C^2(\Om_\infty)$ be a classical solution to
  \eqref {eq:GDETW} on $\Om_\infty$ with the boundary condition \eqref
  {eq:ParabolicBC-top} and the two conditions $s\to s_*$ and
  $k(s)\nabla p\to 0$ as $z\to -\infty$. Then, with $s^*$ as in
  \eqref{eq:DefSstar}, the wave speed satisfies
  \begin{equation}
    c = \left(F_\infty - k(s_*)g L\right )\left/ \left
    (\int^{L}_{0}(s^*(y)-s_*)\, dy\right)\,.\right.
    \label{eq:c-infinite-dom}
  \end{equation}
  If the solution possesses a free boundary, i.e.\,\eqref
  {eq:const-above-h} holds for some $h>0$, then
  \begin{equation}
    g_F := g - \left( F_\infty \left/ {\int_0^L k(s^*(y))\, dy} \right.\right)
    \label{eq:gFdef}
  \end{equation}
  satisfies $g_F > 0$ and there holds $\nabla p(y,z) + g_F e_z \to 0$
  as $z\to \infty$ for every $y\in (0,L)$.
\end{lemma}

\begin{proof}
  Integrating \eqref{eq:Richards2} over $(0,L)\times (-H,H)$ yields
  \begin{align*}
    \left. c \int_0^L s(y,z)\, dy \right|_{z=-H}^H
    =  \left. \int_0^L k(s(y,z))[\del_z p(y,z) + g]\, dy \right|_{z=-H}^H\,. 
  \end{align*}
  Sending $H\to \infty$ provides \eqref {eq:c-infinite-dom}.

  Relation \eqref {eq:const-above-h} implies that $s(y,z) = s^*(y)$
  holds for $z>h$. Therefore, the elliptic equation reduces to
  \begin{equation}
    \nabla\cdot (k(s^*)\nabla p)=0 \text{ in } (0,L)\times (h,\infty)\,.
      \label{eq:EllipticP}
  \end{equation}
  In particular, the flux quantity
  $\si^L_0 k(s^{*})\,\del_z p(y,z)\, dy$ is independent of $z$ for
  $z>h$.  The boundary condition \eqref {eq:ParabolicBC-top} allows to
  evaluate this flux for $z\to \infty$; we find
  \begin{equation}
    \int_0^L k(s^{*}(y))\del_z p(y,z)\, dy = F_\infty - g\int_0^L
    k(s^{*}(y))\, dy = - g_F \int_0^L k(s^{*}(y))\,dy\,.\label{eq:28346}
  \end{equation}
  This provides that, for $z>h$, the weighted average of $\del_z p$
  coincides with $-g_F$.

  Solutions $p$ of the elliptic equation \eqref {eq:EllipticP} with
  homogeneous Neumann boundary conditions on unbounded domains have
  the property that $\nabla p$ stabilizes to a constant as
  $z\to \infty$ (a consequence of the strong maximum principle for
  $\del_z p$).  Relation \eqref {eq:28346} shows that this constant is
  $-g_F e_z$.
  
  Let us assume for a contradiction $g_F<0$. Then $p$ is a growing
  function for $z\to \infty$. This is in contradiction with \eqref
  {eq:closure2}, in which the left hand side vanishes for $z>h$ and
  $p_c(s)$ is independent of $z$ for $z>h$.

  Let us now assume $g_F = 0$ in order to exclude also this case. We
  use a maximum principle for $p$ in the interior of the set
  $\{ (y,z) | \del_z s = 0 \} = \{ (y,z) | p\le p_c(s) \}$.  The
  minimum of $p$ is attained at the boundary.  At the lower boundary
  of this set, there holds $p = p_c(s)$.  This implies that the
  minimum is attained in a point of the form $(y,z) = (y,\Psi(y))$.
  We now use, for any $\e>0$, the strong maximum principle:
  $p(y,\Psi(y)+\e) > p(y,\Psi(y)) = p_c(s(y,\Psi(y))) =
  p_c(s(y,\Psi(y)+\e))$. This implies $p>p_c(s)$ in $(y,\Psi(y)+\e)$
  and hence $\del_z s(y,\Psi(y)+\e) > 0$, in contradiction to the
  construction of $\Psi$.
\end{proof}

\paragraph{Notation.}

Together with the domain $\Om = (0,L)\times \R_+$ with bottom boundary
$\Sigma = (0,L)\times \{0\}$ we also use, for any $H>0$, the bounded
domain $\Om^H := (0,L)\times (0,H)$ with the top boundary
$\Sigma^H := (0,L)\times \{H\}$. We recall that we always impose
homogeneous Neumann conditions at the lateral boundaries
$\{0\} \times \R_+$ and $\{L\} \times \R_+$ (accordingly for the
truncated domain).

The function $\sgn:\R\to \{0,1\}$ is defined as $\sgn(u):=0$ for
$u\leq 0$, and $\sgn(u):=1$ otherwise. The letter $C$ denotes a
generic positive constant and the value may change from one line to
the next in calculations. We already introduced
$[q]_+ = \max\{0,q\} = (q+|q|)/2$ and $[q]_- = \min\{0,q\} = -[-q]_+$.

\section{Existence result for bounded domains}
\label{sec.truncated-domain}

Let $\tau > 0$, $s_*\in (0,1)$, and two functions
$p_0\in H^{\frac{1}{2}}(\Sigma) \cap C^0(\bar{\Sigma})$ and
$s_0\in H^{1}(\Sigma)$ be given. We assume $s_* \le s_0 \le 1$ and
$p_0 \ge p_c(s_0)$.  For a height parameter $H>0$ we introduce the
following truncated problem.

\begin{definition}[Truncated domain travelling wave problem] Let
  $c, F_\infty > 0$ be given.  A pair
  $(s,p)\in H^1(\Om^H)\times H^2(\Om^H)$ on the domain
  $\Omega^H = (0,L) \times (0,H)$ with upper boundary $\Sigma^H$ and
  lower boundary $\Sigma$ is a truncated domain travelling wave
  solution ($TW_H$-solution) if there holds
  \begin{subequations}\label{eq:TTW}
    \begin{align}
      &c\del_z s = \nabla\cdot (k(s)[\nabla p + g e_z])
      &\text{ in } \Om^H\,,
        \label{eq:TWRichards-TTW}\\
      &c\t \del_z s= [p - p_c(s)]_+ &\text{ in } \Om^H\,,\label{eq:HysDyn-TTW}\\
      &s= s_0\,,\ p = p_0 &\text{ on } \Sigma\,,\label{eq:LowerBoundary-TTW}\\
      &p \equiv p^* \in \R &\text{ on } \Sigma^H\,,\label{eq:UpperBoundar-TTW}\\
      &\int_{\Sigma^H} k(s)[\del_z p + g] = F_\infty\,. &\label{eq:Flux-TTW}
    \end{align}
  \end{subequations}
  We emphasize that the constant pressure value $p^*\in \R$ is a free
  parameter and part of the solution of the problem.
\end{definition}

We note that for every $TW_H$-solution $(s,p)$, the flux quantity
\begin{equation}
  \label{eq:flux-z}
  F_c(z) := \int_0^L k(s(y,z))[\del_z p(y,z) + g] - c s(y,z) \ dy
\end{equation}
is independent of $z\in (0,H)$ by \eqref {eq:TWRichards-TTW}.
Evaluating this flux in the upper and in the lower boundary provides,
by \eqref {eq:Flux-TTW},
\begin{equation}
  \label{eq:flux-bottom}
  \int_\Sigma k(s_0)\del_z p
  +  \int_\Sigma \left( k(s_0) g - c s_0 \right)
  = F_\infty - \int_{\Sigma^H} c s\,.
\end{equation}

\begin{remark}
  Let us give a sloppy description of the consequences of \eqref
  {eq:flux-bottom} for small boundary data $s_0$. There is the
  possibility that $\del_z p$ is large at $\Sigma$. This means that a
  sharp transition occurs near the lower boundary. In the opposite
  case (without boundary layer), the left hand side of \eqref
  {eq:flux-bottom} is small. In this case, a moderate flux
  $F_\infty>0$ forces the system that $s$ is not small at
  $\Sigma^H$. This is the desired behavior for finger-like travelling
  wave solutions; they should connect a small saturation at $z=0$ with
  a moderate or large saturation at $z=H$.
\end{remark}

\begin{remark}[A condition for the wave speed $c$]
  Let us highlight another consequence of the fact that $F_c$ of
  \eqref {eq:flux-z} is independent of $z$.  When $(s,p)$ is a
  solution on the doubly unbounded domain $\Om_\infty$ then we expect,
  in the limit $z\to -\infty$, that $s\to s_*$, $p\to p_c(s_*)$, and
  $\del_z p\to 0$. In this situation, the constant flux quantity is
  necessarily $F_c = (g k(s_*) - c s_*) L$.

  We use this observation in order to choose a closure condition for
  the case when the speed $c$ is treated as an unknown: Even when we
  solve a Dirichlet problem in the truncated domain $\Om$ with
  boundary conditions $s_0$ and $p_0$ at the lower boundary $\Sigma$,
  we will seek for $c$ and solutions to the Dirichlet problem that
  satisfy the additional relation
  \begin{equation}
    F_c = \int_{\Sigma} (k(s_0)[\p_z p + g] - c s_0) = (g k(s_*)-cs_*) L\,. 
    \label{eq:ConditionAtSigma}
  \end{equation}
  \Cref{thm:c-F-infty} yields that, given $s_0$, $p_0$, $s_*$, and
  $F_\infty$, we find a speed $c$ such that \eqref
  {eq:ConditionAtSigma} is satisfied.
\end{remark}

In the remainder of this section, we seek for $TW_H$-solutions
$(s,p)$.  We use the space of functions
\begin{align}
  \label{eq:W1q-s-H}
  H^1_\sharp(\Om^H) := \left\{u\in W^{1,2}(\Om^H) \middle|\;  \mathrm{tr}(u) = 0
  \text{ on } \Sigma\,,\ \exists u^* \in \R:\
  u = u^*  \text{ on } \Sigma^H \right\}\,.
\end{align}
The weak formulation of \eqref {eq:TWRichards-TTW} and \eqref
{eq:Flux-TTW} is:
\begin{equation}
  \int_{\Om^H} c\del_z s\, \phi
  + \int_{\Om^H} k(s)[\nabla p+ g e_z]  \cdot \nabla \phi
  = \int_{\Sigma^H} F_\infty\, \phi
  \quad  \text{ for all } \phi\in H^1_\sharp(\Om^H)\,.\label{eq:WeakSolSplus-TTW}
\end{equation}

\begin{theorem}[Existence of $TW_H$-solutions to prescribed data]
  Let $H,c,\t, F_\infty > 0$ and $s_*\in (0,1)$ be given, let
  $p_0\in H^{\frac{1}{2}}(\Sigma)\cap C^0(\bar{\Sigma})$ and
  $s_0\in H^{1}(\Sigma)$ satisfy
  $$
  s_*\leq s_0<1\,,\quad \text{ and }\quad 0< p_0-p_c(s_0) \text{ on }
  \Sigma\,.
  $$
  Then there exists a $TW_H$-solution $(s, p)$ with
  $s, \del_z s\in L^2(\Om^H)$,
  $p\in H^1(\Om^H)\cap H^2_{\mathrm{loc}}(\Om^H)$.
 \label{theo:Wellposedness-TTW}
\end{theorem}

\begin{proof}
  We use an iteration over saturation fields.

  \smallskip {\em Definition of the iteration.} Let there be given a
  saturation field
  $$s^{i-1}\in Y := \left\{ s\in L^2(\Omega^H) \,|\, 
    s_* \le s \le 1 \right\}\,.$$ We define the
  coefficient functions $a := k(s^{i-1})$ and $b := p_c(s^{i-1})$ on
  $\Omega^H$. We seek a solution $p$ of
  \begin{equation}
    \frac1{\tau} [p - b]_+ = \nabla\cdot (a\, [\nabla p + g e_z])
    \text{ in } \Om^H\,,\label{eq:TWRichards-TTW-iter}
  \end{equation}
  with the boundary conditions $p = p_0$ on $\Sigma$ and \eqref
  {eq:UpperBoundar-TTW}--\eqref {eq:Flux-TTW}.  This solution can
  be found with a variational method. We define the space of
  admissible functions as
  $X_{p_0} := \left\{ u\in H^1(\Omega^H) \,|\, u = p_0 \text{ on }
    \Sigma\,,\ \exists u^* \in \R:\ u = u^* \text{ on }
    \Sigma^H\right\}$ and minimize the functional
  \begin{equation}
    \label{eq:A-var-principle}
    A : X_{p_0}\to \R\,,\quad A(p) :=
    \int_{\Omega^H} \frac1{2\tau} [p-b]_+^2 +  \frac12 a\, |\nabla p + g e_z|^2
    - F_\infty \int_{\Sigma^H} p\,.
  \end{equation}
  The functional is convex and coercive, which implies that a
  minimizer $p$ exists. The Euler-Lagrange equation for $p$ reads
  \begin{align*}
    \int_{\Omega^H} \frac1{\tau} [p-b]_+\, \fhi + a\, [\nabla p + g e_z]
    \cdot \nabla \fhi
    = F_\infty \int_{\Sigma^H} \fhi \qquad\quad
    \forall \fhi\in H^1_\sharp(\Om^H)\,.
  \end{align*}
  Since arbitrary compactly supported test-functions $\fhi$ can be
  inserted, equation \eqref {eq:TWRichards-TTW-iter} holds for
  $p$. The Euler-Lagrange equation additionally encodes the boundary
  condition $\int_{\Sigma^H} a\, (\del_z p + g) = F_\infty$.  
  Given $p^i = p$, we can solve the family of ordinary differential
  equations
  \begin{equation}\label{eq:ODE-pf1}
    c\t \del_z s= [p^i - p_c(s)]_+ \,,
  \end{equation}
  with initial data $s= s_0$ on $z=0$; this system is related to
  \eqref {eq:HysDyn-TTW} together with the first equation in
  \eqref{eq:LowerBoundary-TTW}. We denote the solution of this system
  by $s =: s^i$.

  \smallskip {\em Fixed point of the iteration.} We claim that, for
  some constant $C = C(H,c,\tau)$ independent of $s^{i-1}$, the
  pressure $p = p^i$ satisfies
  \begin{equation}
    \label{eq:iter-press-bound-L2}
    \| p \|_{L^2(\Omega^H)}^2 + \| \nabla p \|_{L^2(\Omega^H)}^2 \le C\,.
  \end{equation}
  In order to show this estimate, we first choose an $H^1$-extension
  $\hat p_0$ of the data $p_0$, vanishing at the upper boundary. We
  can now multiply equation \eqref {eq:TWRichards-TTW-iter} with
  $p - \hat p_0$ and integrate to obtain
  \begin{align*}
    &\int_{\Omega^H} \frac1{\tau} [p - b]_+ \left( [p-b] - \hat p_0 + b\right)
      + \int_{\Omega^H} (a\, [\nabla p + g e_z])\cdot \nabla (p - \hat p_0)
      = \int_{\Sigma^H}  F_\infty\, p\,.
  \end{align*}
  One of the integrals on the left hand side is an upper bound for
  $k(s_*) \| \nabla p \|_{L^2(\Omega^H)}^2$, the other term with
  quadratic growth in $p$ is on the left hand side and positive
  because of $[p - b]_+ [p-b] \ge 0$. The remaining terms have linear
  growth in $p$ and can therefore be estimated with Youngs inequality
  and with the Poincar\'e inequality.

  The corresponding solutions $s^i = s$ of the ordinary differential
  equation satisfy $0\le s\le 1$ by the growth assumption on $p_c$. In
  particular, there holds $s^i\in Y$. With
  $R := |\Om^H|^{1/2} = |LH|^{1/2}$, we find that the above
  construction provides a map
  $$\calT : Y \supset B_R(0) \to B_R(0) \subset Y\,,
  \qquad s^{i-1} \mapsto s^i\,.$$

  We claim that the map $\calT$ is compact.  We will show the
  compactness below with the characterization of compact subsets of
  $L^2(\Om^H)$ by Kolmogorov-Riesz. An application of Schauder's fixed
  point theorem yields the existence of the desired solution $s$.

  Let us turn to compactness of $\calT$.  We consider the family
  $p = p^i$ of solutions for $s = s^{i-1} \in B_R(0)$. This family of
  solutions is bounded in $H^1(\Om^H)$, hence the finite differences
  $p(y,.) - p(y + \delta, .)\in L^2((0,H);\R)$ are small for
  $\delta >0$ small, independent of $s$. More precisely,
  $$\int_0^{L-\delta} \int_0^H |p(y,z) - p(y + \delta, z)|^2\, dz\, dy
  \le \eta(\delta)\,,$$ with $\eta(\delta) \to 0$ as $\delta\to 0$,
  independent of $s$.  We now consider two solutions of the ordinary
  differential equation \eqref {eq:ODE-pf1}, $s(y,.)$ and
  $s(y + h, .)$ to inputs $p(y,.)$ and $p(y + h, .)$.  The solutions
  differ only as much as their right hand sides and their initial
  values differ. Because of our assumption $s_0\in H^1(\Sigma)$, we
  therefore find also for the solutions
  $$\int_0^{L-\delta} \int_0^H |s(y,z) - s(y + \delta, z)|^2\, dz\, dy
  \le C \eta(\delta)\,.$$ On the other hand, since $\del_z s$ is
  bounded in $L^2(\Om^H)$, the corresponding estimate
  $\int_0^{L} \int_0^{H-\delta} |s(y,z+\delta) - s(y, z)|^2\, dz\, dy
  \le C \eta(\delta)$ is clear. This shows compactness of the image
  set of $s$-fields.  
\end{proof}

\section{Unbounded domain solutions for $H\to \infty$}
\label{sec.unbounded-domain}

In this section we analyze the solutions $(s_H,p_H)$ in the limit
$H\to 0$.  Again, for the larger part of this section, we keep
$\tau>0$, $s_*\in (0,1)$, $F_\infty$, and $c>0$ fixed; only in Theorem
\ref {thm:c-F-infty} we determine $c$ from the other parameters. The
main result of this section is the following: Let $(s_H,p_H)$ denote
the $TW_H$-solution as discussed in \Cref
{theo:Wellposedness-TTW}. Then, for $H\to \infty$, there holds
$(s_H,p_H) \to (s,p)$ in an appropriate sense for some limit pair
$(s,p)$, which is defined on the unbounded domain $\Omega$. The pair
$(s,p)$ is a travelling wave solution for the semi-infinite domain $\Om$.

It turns out that two different limiting solution types are possible.
Type I is the ``large solution''. It is characterized by the following
properties: 1) The solution is large in the sense that
$\int_0^L g k(s(y,z_0))\, dy \ge F_\infty$ for some $z_0$. This means
that a certain $F_\infty$-dependent threshold is exceeded by the
saturation variable. 2) The solution has a free boundary: For some
$h>0$ there holds $\del_z s(y,z) = 0$ for every $z\ge h$. 3) The
solution has an unbounded pressure, $p\to -\infty$ as $z\to \infty$.

Accordingly, Type II solutions are the ``small solutions''. They have
a bounded pressure and no free boundary.

To proceed with the analysis, we consider different assumptions.

\begin{assumption}\label{ass:unbounded-dom}
  The following properties can be considered for the solution sequence
  $(s_H,p_H)$ of \eqref{eq:TTW}, obtained in
  \Cref{theo:Wellposedness-TTW}.
  \begin{description}
  \item[Bounds for parameters] The limiting saturation $s^*\in (0,1)$,
    the wave speed $c$, and the flux $F_\infty$ satisfy
    \begin{subequations}
      \begin{align}
        gk'(s_*) <\, &c < g(k(1)-k(s_*))/(1-s_*)\,,\\
        g L [k(s_*) + k'(s_*)(1-s_*)] <\, & F_\infty < gL k(1)\,.
      \end{align} \label{eq:FingrCF}
    \end{subequations}\vspace*{-8mm}
  \item[Bound for the pressure] For a real number $\bar{p}<\infty$
    independent of $H$ holds
    \begin{equation}
      p_H \leq \bar{p}\quad\text{ in } \Omega^H\,. 
      \label{eq:FingrP}
    \end{equation}
  \item[Local bound for the gradient] There exists $C_P>0$ such that,
    for every $H>0$,
    \begin{equation}
      \|\nabla p_H\|_{L^\infty(\Om^H)}\leq C_P\,. 
      \label{eq:FingrDP}
    \end{equation}
  \item[Regularity] The saturation has the regularity properties
    \begin{equation}
      s_H,\,\del_z s_H\in H^1(\Om^H)\,.
      \label{eq:ass-s-regularity}
    \end{equation}
  \end{description}
\end{assumption}

The assumptions have a quite different character.  Inequalities \eqref
{eq:FingrCF} are ranges for the physical parameters; we expect the
existence of travelling waves in this parameter regime.  The uniform
upper bound of \eqref{eq:FingrP} is expected to hold, but it should be
derived from the system of equations, which we did not succeed to do.
The regularity estimate \eqref {eq:FingrDP} and the local
regularity \eqref {eq:ass-s-regularity} can be shown with the tools of
elliptic regularity theory, see \cite {gilbarg2015elliptic}.  We
formulate them here as assumptions, since the regularity theory is not
the focus of this contribution.

We note that the relations \eqref{eq:FingrP}--\eqref{eq:FingrDP} imply
three further estimates:
\begin{subequations}
\begin{equation}
  \|s_H\|_{L^\infty(\Om^H)} \leq \bar{s} := {p_c}^{-1}(\bar{p}) < 1\,.
  \label{eq:FingrSm}
\end{equation}
In \Cref{lem:NablaSH} we prove that, for a constant
$C_s=C_s(C_P,s_0,p_0)$,
\begin{equation}
  \|\nabla s_H\|_{L^\infty(\Om^H)} \leq C_s\,.
  \label{eq:pzSlowerBOund}
\end{equation}
Since \eqref{eq:pzSlowerBOund} provides
$\|\p_z s_H\|_{L^\infty(\Om^H)}< C_s$, one also has from
\eqref{eq:HysDyn-TTW} that
\begin{equation}
  p_H \leq p_c(s_H) + c\t C_s \quad \text{ in } \Om^H.
  \label{eq:PminPc_BOund}
\end{equation}
\end{subequations}

Our main result on unbounded domains is the following.

\begin{theorem}[Limits of $TW_H$-solutions]
  \label{theo:main}
  Let $c, F_\infty, \t>0$, $s_*\in (0,1)$, and boundary data
  $s_0,\,p_0\in C^1(\Sigma)$ with $s_*\le s_0< 1$ and $p_c(s_0)< p_0$
  be given.  Let all the properties of Assumption \ref
  {ass:unbounded-dom} be satisfied.  For a sequence $H\to \infty$, let
  $(s_H,p_H)$ be solutions to \eqref{eq:TTW}.  Then, for a limiting
  pair $(s,p)$, there holds $(s_H,p_H) \to (s,p)$ locally in
  $L^2(\Omega)$. The limits satisfy $s\in C^0_b(\Om)$,
  $\del_z s\in L^2(\Om)$,
  $p\in H^2_{\mathrm{loc}}(\Om) \cap H^1_{\mathrm{loc}}(\Om\cup
  \Sigma)$, $(s,p)=(s_0,p_0)$ on $\Sigma$, and \eqref{eq:GDETW}.  The
  solution $(s,p)$ is either of Type I or of Type II:
  \begin{description}
  \item[Type I: ``Large solution''] The solution has a free-boundary:
    There exists $h\in \R_+$ such that $\del_z s=0$ for all
    $y\in (0,L)$ and $z\geq h$. The solution is large in the sense that, with
    $s^*(y) := \lim_{z\to \infty} s(y,z)$, there holds
    $g \int_0^L k(s^*(y))\, dy \geq F_\infty$, with strict inequality
    if $p$, $s$, and $\del_z s$ are continuous. Furthermore, $p(y,z)\to -\infty$ as
    $z\to \infty$ in this case.
  \item[Type II: ``Small solution''] The solution has a bounded
    pressure, there holds $p\in L^\infty(\Om)$. Furthermore,
    $\nabla p\in L^2(\Om)$.  The solution is ``small'' in the sense
    that $g \int_0^L k(s^*(y))\, dy \le F_\infty$. 
  \end{description}
  Type I solutions satisfy additionally the boundary condition \eqref
  {eq:ParabolicBC-top}.
\end{theorem}

The theorem follows from Propositions \ref {prop:SolCase1exists} and
\ref {prop:SolCase2exists}. Before we can prove these results, we have
to establish an a priori estimate, which is the basis for both
propositions.

\begin{lemma}[A priori estimate for $TW_H$-solutions]
  Let $F_\infty, c, \t, s_* > 0$ and $s_0,p_0\in C^1(\Sigma)$ with
  $s_*\leq s_0(y)<1$ and $p_c(s_0)< p_0$. For a sequence
  $0 < H\to \infty$, let $(s_H,p_H)$ be solutions to
  \eqref{eq:TTW}. We assume that the solution sequence satisfies
  relations \eqref{eq:FingrDP} and \eqref {eq:ass-s-regularity}.  We
  use the characteristic functions
  $\cut_{\!_>} :=\cut_{\{\del_z s_H>0\}}$ and
  $\cut_0:=\cut_{\{\del_z s_H =0\}}$ on $\Om^H$.  There exists a
  constant $C_1 := C_1(c,\t,s_0,p_0,C_P)$, independent of $H$, such
  that
  \begin{subequations}\label{Eq:MainAprioriEstimate}
    \begin{equation}\label{Eq:MainAprioriEstimate1}
      \int_{\Om^H}  \cut_{\!_>}\, {p_c}'(s_H)|\nabla s_H|^2
      + \int_{\Om^H} \cut_0\,  \tfrac{1}{{p_c}'(p_c^{-1}(p_H))}|\nabla p_H|^2
      + c\t\int_{\Sigma^H} |\nabla s_H|^2 \leq C_1\,.
    \end{equation}
    If, additionally, \eqref{eq:FingrP} is satisfied, there exists
    $C_2:=C_2(c,\t,s_0,p_0,C_P,\bar{p})$ such that
    \begin{equation}\label{Eq:MainAprioriEstimate2}
      \int_{\Om^H}  \cut_{\!_>}|\nabla p_H|^2
      + \int_{\Om^H} |\nabla (\del_z s_H)|^2 \leq C_2\,.
    \end{equation}
  \end{subequations}
  \label{lem:AprioriEstimates}
\end{lemma}

\begin{proof} Within this proof, we write $(s,p)$ instead of
  $(s_H,p_H)$ to have shorter formulas.  With $C>0$ we refer to
  generic constants that may depend on $c,\t,s_0,p_0,C_P,\bar{p}$, but
  not on $H$.

  \smallskip {\em Step 1: Test function $K(s)$.} We use
  $K:[0,1]\to [0,\infty)$, defined as
  $K(s) := \int_0^s k(\vr)^{-1}\, d\vr$. Equivalently, we may say that
  $K$ is the primitive of $k^{-1}$, satisfying
  \begin{equation*}
    K'(s)=\frac{1}{k(s)}\,,\quad K(0)=0\,.
  \end{equation*}
  Below, we will use additionally the primitive of $K$; we denote by
  $\tilde{K}$ the function that satisfies $\tilde{K}'(s)=K(s)$ and
  $\tilde{K}(0)=0$.

  We use $K(s)(y,z) = K(s(y,z))$ as a test function in
  \eqref{eq:TWRichards-TTW} and study
  $$c\int_{\Om^H} K(s) \del_z s= \int_{\Om^H} K(s) \nabla\cdot
  (k(s)[\nabla p + g e_z])\,.$$ Using an integration by parts, we
  may write this relation as
  \begin{align}
    &c \int_{\Om^H}\del_z \tilde{K}(s)
      + \int_{\Om^H} k(s)[\nabla p + g e_z]\cdot \nabla K(s)\nonumber\\
    &\qquad
      = \int_{\Sigma^H} K(s) k(s)[\del_z p + g]
      - \int_{\Sigma} K(s_0)  k(s_0)[\del_z p + g]\,. 
  \end{align}
  We have constructed $K$ such that $\nabla K(s) = k(s)^{-1}\nabla
  s$. This gives a simple formula for the second integral. With
  another integration by parts and with
  $\cut = \cut_{\!_>} + \cut_{0}$ we find
    \begin{align}
      &c \int_{\Sigma^H} \tilde{K}(s)- c \int_{\Sigma}\tilde{K}(s_0)
      + \int_{\Om^H}\cut_{\!_>}\, \nabla p \cdot \nabla s
      + \int_{\Om^H} \cut_{0}\, \p_y p\, \p_y s + \int_{\Om^H} g e_z \cdot \nabla s\nonumber\\
      &= \int_{\Sigma^H} K(s)\, k(s)\,\del_z p
      - \int_{\Sigma} K(s_0)  k(s_0)\del_z p +  g\int_{\Sigma^H} K(s)\, k(s)
      - g\int_{\Sigma} K(s_0)  k(s_0)\,. 
      \label{eq:AprioriAllTerms}
    \end{align}
    We note that the last two integrals on the right hand side and the
    first two integrals on the left hand side are bounded.  Since we
    assumed \eqref{eq:FingrDP}, actually the entire right hand side of
    \eqref{eq:AprioriAllTerms} is bounded. The last integral of the
    left hand side can be integrated, which shows that also this term
    is bounded. We therefore find 
    \begin{align}
      &\int_{\Om^H}\cut_{\!_>}\, \nabla p \cdot \nabla s
      + \int_{\Om^H} \cut_{0}\, \p_y p\, \p_y s\le C\,.
      \label{eq:AprioriAllTermB}
    \end{align}
    
    We want to rewrite the first integral. With this aim, we observe
    that $c\t\del_z s= [p-p_c(s)]_+$ in $\Om^H$ implies $c\t \nabla \del_z
    s= (\nabla p- {p_c}'(s)\nabla s)\cut_{\!_>}$ (we recall that we
    assumed $\del_z s\in H^1(\Om)$). This yields
    \begin{equation}
      \int_{\Om^H} \cut_{\!_>} \nabla p \cdot \nabla s
      =c\t\int_{\Om^H} \nabla s\cdot \nabla \del_z s
      + \int_{\Om^H} \cut_{\!_>}\, p_c'(s)|\nabla s|^2\,.\label{eq:Cut1ps}
    \end{equation}
    The first term on the right hand side of \eqref{eq:Cut1ps} is
    \begin{align}
      &c\t\int_{\Om^H} \nabla s\cdot \nabla \del_z s= c\t\int_{\Om^H}
      \del_z\left (\frac{1}{2}|\nabla s|^2\right )=\frac{c\t}{2}
      \int_{\Sigma^H} |\nabla s|^2 - \frac{c\t}{2} \int_{\Sigma}
      |\nabla s|^2 \nonumber\\
      &\qquad =\frac{c\t}{2} \int_{\Sigma^H}
      |\nabla s|^2 - \frac{1}{2c\t} \int_{\Sigma} [p_0-p_c(s_0)]_+^2
      - \frac{c\t}{2} \int_{\Sigma}|\del_y s_0|^2\,.
    \end{align}
    At this point, we obtained from \eqref{eq:AprioriAllTermB} 
    \begin{equation}
      \int_{\Om^H} \cut_{0}\, \p_y s\,\p_y p + \int_{\Om^H}
      \cut_{\!_>}\, p_c'(s)|\nabla s|^2+
      \frac{c\t}{2}\int_{\Sigma^H} |\nabla s|^2 \leq C \,.
      \label{eq:FirstStepApriori}
    \end{equation}
    
    \smallskip
    \textit{Step 2: Test function $\Phi$.}  We next consider the new
    test function
    \begin{equation*}
      \Phi:=[K(s)-K({p_c}^{-1}(p))]_+\in H^1(\Om^H)\,. 
    \end{equation*}
    Note that $\del_z s>0$ $\iff$ $p>p_c(s)$ $\iff$ ${p_c}^{-1}(p)>s$
    $\iff$ $K({p_c}^{-1}(p))>K(s)$. This shows
    $$\Phi=[K(s)-K({p_c}^{-1}(p))]\cut_0\,.$$
    Using $\Phi$ as a test function for \eqref{eq:TWRichards-TTW} and
    exploiting that $\Phi\neq 0$ only when $\del_z s =0$, we find
    \begin{equation}
      \int_{\Om^H} \Phi \nabla\cdot (k(s)[\nabla p + ge_z])
      = c\int_{\Om^H} \Phi \del_z s
      =  0\,.
      \label{eq:TestFucn2_1}
    \end{equation}
    Also on the left hand side, the term $\Phi \nabla\cdot[k(s) ge_z]
    = \Phi k'(s) g \del_z s$ vanishes identically.  Integration by
    parts in \eqref{eq:TestFucn2_1} yields, using $\Phi=0$ on
    $\Sigma$,
    \begin{equation}
      \int_{\Om^H} k(s)\nabla \Phi \cdot \nabla p
      = \int_{\Sigma^H} \Phi k(s) \del_z p\,.\label{eq:TestFucn2}
    \end{equation}
    Because of $\nabla \Phi=\left (\tfrac{1}{k(s)}\nabla s -
    \tfrac{1}{k({p_c}^{-1}(p))}\frac{1}{p_c'({p_c}^{-1}(p))}\nabla p
    \right )\cut_0$, we find
    \begin{equation}
      \int_{\Om^H} \nabla s\cdot \nabla p \,\cut_0
      - \int_{\Om^H}\frac{k(s)}{k({p_c}^{-1}(p))}\
      \frac{|\nabla p|^2}{p_c'({p_c}^{-1}(p))}\, \cut_0
      = \int_{\Sigma^H} \Phi \,k(s) \del_z p\,.
    \end{equation}
    The first integral is
    $\int_{\Om^H} \nabla s\cdot \nabla p \,\cut_0= \int_{\Om^H} \p_y
    s\, \p_y p \,\cut_0$, hence it coincides with the first term in
    \eqref{eq:FirstStepApriori}. Since
    $k(s)\cut_0> k({p_c}^{-1}(p))\cut_0$, from
    \eqref{eq:FirstStepApriori} we arrive at
    \begin{align}
      & \int_{\Om^H}\tfrac{1}{p_c'({p_c}^{-1}(p))}\, |\nabla p|^2\, \cut_0
      + \int_{\Om^H}  \cut_{\!_>}\, {p_c}'(s)|\nabla s|^2
      + \frac{c\t}{2}\int_{\Sigma^H} |\nabla s|^2 \leq C\,,
      \label{eq:SecondStepApriori}
    \end{align}
    where we exploited once more \eqref{eq:FingrDP}.  At this point,
    we have shown \eqref{Eq:MainAprioriEstimate1}.
    
    \smallskip \textit{Step 3: Test function $\del_z s$.} To show
    \eqref{Eq:MainAprioriEstimate2}, we use the test function
    $\del_z s = \frac{1}{c\t}[p-p_c(s)]_+\in H^1(\Om^H)$ in
    \eqref{eq:TWRichards-TTW}. With an integration by parts we obtain
    \begin{equation}
    \begin{split}
      &\frac{1}{c\t}\int_{\Om^H} k(s)\nabla p\cdot \nabla
      [p-p_c(s)]_+\\
      &\qquad =\int_{\Sigma^H} \del_z s\, k(s)\del_z p - \int_{\Sigma}
      \del_z s\, k(s_0)\del_z p + \int_{\Om^H} (gk'(s)-c)|\del_z s|^2\,.
    \end{split}
    \label{eq:Step3-res1}
    \end{equation}
    We observe that, by \eqref{eq:FingrDP} and
    \eqref{eq:pzSlowerBOund}, the first two integrals on the right
    hand side are bounded. Furthermore, the middle term of
    \eqref{eq:SecondStepApriori} shows that also the last integral is
    bounded.

    Using the algebraic manipulation $2a(a-b)=a^2-b^2 + (a-b)^2$, the
    left hand side of \eqref {eq:Step3-res1} is written as
    \begin{align*}
      &\frac{1}{c\t}\int_{\Om^H} k(s)\nabla p\cdot \nabla [p-p_c(s)]_+
      = \frac{1}{c\t}\int_{\Om^H} k(s)\nabla p\cdot (\nabla p-\nabla p_c(s)) \cut_{\!_>}\\
      &\quad = \frac{1}{2c\t}\int_{\Om^H} k(s)[|\nabla p|^2 +|\nabla (p-p_c(s))|^2
        - |\nabla p_c(s)|^2]  \cut_{\!_>}\\
      &\quad = \frac{1}{2c\t}\int_{\Om^H} k(s)[\cut_{\!_>} |\nabla p|^2
        +(c\t)^2 |\nabla (\del_z s)|^2 - \cut_{\!_>}({p_c}'(s))^2 |\nabla s|^2]\,.
    \end{align*}
    Inequality \eqref{eq:SecondStepApriori} along with \eqref
    {eq:FingrSm} shows that the negative term has a bounded integral.
    This shows \eqref {Eq:MainAprioriEstimate2} and concludes the
    proof.
\end{proof}

To investigate the free-boundary structured solution described in
\Cref{theo:main} we define the function $h:\R_+\to \R_+$ with
\eqref{eq:const-above-h} in mind: For $H>0$ and $(s_H,p_H)$ solving
\eqref{eq:TTW}, $h=h(H)$ is defined as
\begin{align}
  h(H) := \inf\{z_0\in [0,H]: \del_z s_H = 0
  \text{ a.e. in } (0,L)\times (z_0,H) \}\,.
  \label{eq:DefhH}
\end{align}
The height $h$ marks a horizontal line such that, above that line,
$\del_z s$ vanishes. We note that $h\in [0,H]$ is well-defined and
that $h = H$ is possible.

\begin{proposition}[Free-boundary solutions]
  \label{prop:SolCase1exists}
  We consider the situation of \Cref{theo:main} with a sequence
  $(s_H,p_H)$ of $TW_H$-solutions for $H\to \infty$. Additionally, we
  assume for the sequence $H\to \infty$ that the height
  \begin{equation}
    h(H)\quad\text{is bounded.} 
  \end{equation}
  Under this assumption, a free-boundary travelling wave solution
  $(s, p)$ exists.  More precisely, there exists a pair $(s,p)$ with
  $s\in C^0_b(\Om)$, $\del_z s\in L^2(\Om)$,
  $p\in H^2_{\mathrm{loc}}(\Om) \cap H^1_{\mathrm{loc}}(\Om\cup
  \Sigma)$, satisfying \eqref{eq:GDETW}--\eqref{eq:ParabolicBC}. The
  solution is of free boundary type in the sense that there exists
  $h^*>0$ such that $\del_z s=0$ for all $y\in (0,L)$ and $z\geq
  h^*$. The flux satisfies
  \begin{equation}\label{eq:F-infty-large-sol}
    F_\infty \le g\int^L_0 k(s^*(y))\, dy\,.
  \end{equation} 
  Under the additional regularity assumptions
  $s, \del_z s, p \in C^0(\Om)$, the strict inequality holds in \eqref
  {eq:F-infty-large-sol}.
\end{proposition}

\begin{proof}
  Let $h^*>0$ denote an upper bound of the function $h(H)$, i.e.
  \begin{equation}
    h(H)\leq h^* \text{ for all } H.
  \end{equation}

  \smallskip {\em Step 1: An additional a priori estimate.} 
  We consider once more the function
  \begin{equation}
    s^*_H(y)=s_H(y,h^*) \text{ for all } y\in (0,L)\,.
  \end{equation}
  Let $g_{\!_{F,H}}\in \R$ be the number
  \begin{equation}\label{eq:gFH-def}
    g_{\!_{F,H}} := g - \left(F_\infty \middle/ \int_0^L k(s^*_H(y))\, dy\right)\,,
  \end{equation}
  and let $\tilde{p}_H\in H^1(\Om^H)$ be the function
  \begin{equation}
    \tilde{p}_H(y,z) := p_H(y,z) + g_{\!_{F,H}} z \text{ for } (y,z)\in \Om^H\,.
  \end{equation}
  We note that these definitions reflect the observations of
  \Cref{lem:WaveSpeed-Ominfty}. We finally define
  $\fhi_H\in C^2([0,1])$ as the function
  \begin{equation}
    \fhi_H(s) := c s - (g-g_{\!_{F,H}}) k(s)\,.
  \end{equation}
  This allows to write \eqref{eq:TWRichards-TTW} in the form
  \begin{equation}
    \nabla\cdot[k(s_H)\nabla \tilde{p}_H] = \del_z
    \fhi_H(s_H)\,.\label{eq:RicharsDivForm}
  \end{equation}
  We observe that, by \eqref {eq:Flux-TTW} and the choice of
  $g_{\!_{F,H}}$ in \eqref{eq:gFH-def},
  \begin{equation}
    \label{eq:pf-fb-824}
    \int_{\Sigma^H}\, k(s_H)\del_z \tilde{p}_H
    = F_\infty + (g_{\!_{F,H}} - g)  \int_{\Sigma^H}\, k(s_H)
    =0\,.
  \end{equation}
  The test function $\tilde{p}_H$ in \eqref{eq:RicharsDivForm}
  provides the identity
  \begin{equation}\label{eq:pf-fb-235}
    \int_{\Om^H} \tilde{p}_H\, \nabla\cdot[k(s_H)\nabla
    \tilde{p}_H]=\int_{\Om^H}\tilde{p}_H\, \del_z \fhi_H(s_H)\,.
  \end{equation}
    
  The left hand side of \eqref {eq:pf-fb-235} is calculated with an
  integration by parts, exploiting the fact that $p_H$ on the upper
  boundary is constant, $p_H \equiv p_H^*$ on $\Sigma_H$. In the last
  line of the calculation we use \eqref {eq:pf-fb-824}.
  \begin{align*}
    &\int_{\Om^H} \tilde{p}_H \nabla\cdot[k(s_H)\nabla \tilde{p}_H]\\
    &\quad =
     - \int_{\Om^H} k(s_H)|\nabla \tilde{p}_H|^2
     + \int_{\Sigma^H} \tilde{p}_H k(s_H)\del_z \tilde{p}_H
     -\int_{\Sigma} \tilde{p}_H k(s_H)\del_z \tilde{p}_H\\
    &\quad = - \int_{\Om^H} k(s_H)|\nabla \tilde{p}_H|^2
     + (p^*_H + g_{\!_{F,H}} H)\int_{\Sigma^H} k(s_H)\del_z \tilde{p}_H
     - \int_{\Sigma} p_0\, k(s_0) [\del_z p_H + g_{\!_{F,H}}] \\
    &\quad = - \int_{\Om^H}
    k(s_H)|\nabla \tilde{p}_H|^2 -
    \int_{\Sigma} p_0\, k(s_0)[\del_z p_H + g_{\!_{F,H}}] \,.
  \end{align*}
  The right hand side of \eqref {eq:pf-fb-235} is treated with two
  integrations by parts, 
  \begin{align*}
    &\int_{\Om^H}\tilde{p}_H\, \del_z \fhi_H(s_H)\\
    &\quad =- \int_{\Om^H}\fhi_H(s_H) \del_z\tilde{p}_H 
    + \int_{\Sigma^H}\tilde{p}_H\, \fhi_H(s^*_H)-
    \int_{\Sigma}\tilde{p}_H\, \fhi_H(s_H) \\
    &\quad = - \int_{\Om^H}\fhi_H(s_H) \del_z \tilde{p}_H
    + \left[ \int_{\Om^H}\del_z \tilde{p}_H\,
      \fhi_H(s^*_H) + \int_{0}^L p_0\, \fhi_H(s^*_H) \right] -
    \int_{\Sigma} p_0 \, \fhi_H(s_0)\\
    &\quad = \int_{\Om^H}(\fhi_H(s^*_H) - \fhi_H(s_H))\del_z
    \tilde{p}_H+ \int_{0}^L(\fhi_H(s^*_H) - \fhi_H(s_0))p_0\,.
  \end{align*}
  Boundedness of many of the above terms can be concluded from the
  facts that $g_{\!_{F,H}}$ is bounded, $\fhi_H\in C^1([0,1])$, and
  boundedness of $\del_z p$ from \eqref {eq:FingrDP}.  From \eqref
  {eq:pf-fb-235} and Young's inequality we obtain
  \begin{align*}
    &\int_{\Om^H} k(s_H)|\nabla \tilde{p}_H|^2
    \leq C - \int_{\Om^H}(\fhi_H(s^*_H) - \fhi_H(s_H))\del_z \tilde{p}_H\\
    &\quad \leq C + \int_{\Om^H}\frac{1}{2k(s_H)}|\fhi_H(s^*_H) - \fhi_H(s_H)|^2
    + \int_{\Om^H} \frac{k(s_H)}{2}|\del_z \tilde{p}_H|^2\,.
  \end{align*}
  We have applied Young's inequality in such a way that the last term
  on the right hand side can be substracted from both sides.  Since
  $\fhi_H(s^*_H) - \fhi_H(s_H)=0$ holds for $z\geq h^*$, the first
  integral on the right hand side is bounded.  We conclude
  \begin{align*}
    &\int_{\Om^H} k(s_H) |\nabla \tilde{p}_H|^2 \leq C (1 + h^*)\,.
  \end{align*}
  Recalling additionally the estimates from \Cref
  {lem:AprioriEstimates}, we have the following estimates for the
  solution sequence:
  \begin{align}
    \int_{\Om^H}[ |\nabla \tilde{p}_H|^2
    + |\del_z s_H|^2 + |\nabla \del_z s_H|^2]
    \le C\,.\label{eq:PtildeSbound}
  \end{align}

  \smallskip {\em Step 2: Limit equations.}  It remains to exploit the
  bounds of \eqref {eq:PtildeSbound} to construct the limit solution
  for $H\to \infty$.  Since the sequence $g_{\!_{F,H}}$ is bounded, we
  can choose a subsequence $\{H_i\}_{i\in \N}$ with $\lim H_i=\infty$
  and $g_F\in \R$ such that $g_{\!_{F,H_i}}\to g_{F}$. In the
  following, we only use this subsequence. The estimate
  \eqref{eq:PtildeSbound} allows to choose a further subsequence and a
  pair $(s,p)$ with $s\in H^1_{\mathrm{loc}}(\Om) \cap L^\infty(\Om)$
  and $p\in H^1_{\mathrm{loc}}(\Om)$ such that, for any bounded
  compact subset $\Om'\subset \Om$, there holds
  \begin{subequations}\label{eq:Convergence_sHpH}
  \begin{align}
    &s_{H} \to s\text{ and }
    \del_z s_{H} \to \del_z s \text{ strongly in } L^2(\Om')\,,\\
    &p_{H} \rightharpoonup p \text{ weakly in } H^1(\Om')\text{ and }
    p_{H} \to p \text{ strongly in } L^2(\Om')\,.
  \end{align}
  \end{subequations}
  These convergences imply that also the limit $(s,p)$ satisfies
  \eqref{eq:GDETW} in $\Om$ and the boundary conditions at the lower
  boundary. Furthermore, $\del_z s_H \equiv 0$ for all $H$ on
  $\{ z\ge h^*\}$ implies $\del_z s \equiv 0$ on $\{ z\ge h^*\}$.
  
  \smallskip {\em Step 3: Flux relations.}  Regarding the flux we use
  that the quantity
  \begin{equation}
    \label{eq:flux-H-z}
    F_c^H(z) := \int_0^L k(s_H(y,z)) [\del_z p_H(y,z) + g] - c s_H(y,z)\, dy
  \end{equation}
  is independent of $z\ge 0$ (compare $F_c$ in \eqref
  {eq:flux-z}). Since the saturation $s_H$ is independent of $z$ for
  $z\ge h^*$, also the quantity
  \begin{equation}
    \label{eq:flux-F-z}
    F^H(z) := \int_0^L k(s_H(y,z)) [\del_z p_H(y,z) + g]\, dy
  \end{equation}
  is independent of $z$ for $z\ge h^*$.  Because of this independence
  and because of $F^H(H) = F_\infty$, we find, as $H\to \infty$, for
  every $z\ge h^*$,
  \begin{equation}
    \label{eq:flux-F-z-lim}
    F_\infty = F^H(z) \to \int_0^L k(s(y,z)) [\del_z p(y,z) + g]\, dy\,.
  \end{equation}
  This shows that the boundary condition \eqref {eq:ParabolicBC-top}
  is satisfied by the limit functions.

  \smallskip We have found a free boundary solution on an unbounded
  domain.  As in Lemma \ref {lem:WaveSpeed-Ominfty}, there follows
  $g_F \ge 0$ and, under the regularity assumptions
  $s, \del_z s, p \in C^0(\Om)$, the strict inequality $g_F > 0$. This
  implies
  $F_\infty = (g-g_F)\int^L_0 k(s^*(y))dy \le g\int^L_0 k(s^*(y))\,
  dy$, and hence \eqref {eq:F-infty-large-sol}.
\end{proof}

\begin{proposition}[Bounded pressure solutions]\label{prop:SolCase2exists}
  Let the situation be that of \Cref{theo:main}, with $TW_H$-solutions
  $(s_H,p_H)$ along a sequence $H\to \infty$. We assume here that
  the sequence of heights  $h(H)$ diverges,
  \begin{equation}
    h(H)\to \infty\quad\text{as}\quad H\to \infty\,.
  \end{equation}
  Then, a bounded pressure travelling wave solution $(s, p)$ exists.
  More precisely, there exists a pair $(s,p)$ with $s\in C^0_b(\Om)$,
  $\del_z s\in L^2(\Om)$,
  $p\in H^2_{\mathrm{loc}}(\Om) \cap H^1_{\mathrm{loc}}(\Om\cup
  \Sigma)$ satisfying \eqref{eq:GDETW}.  For $C>0$ there holds
  $$ \|p\|_{L^\infty(\Om)} + \|\nabla p\|_{L^2(\Om)} + \|\del_z
  s\|_{H^1(\Om)} \leq C\,.$$ The solution satisfies
  \begin{equation}
    g \int^L_0 k(s^*(y))\, dy \le F_\infty\,.
  \end{equation}
\end{proposition}

We note that we do not obtain the flux condition \eqref
{eq:ParabolicBC}.

\begin{proof} In this proof, we only write $H\to \infty$
  and $h\to \infty$ for the two sequences. We furthermore use
  $\Om^h=[0,L]\times (0,h)$.

  \smallskip {\em Step 1: $L^\infty$-bound for the pressure.}  The
  upper bound for the pressure was assumed in \eqref{eq:FingrP},
  $p_H\le \bar p$ in $\Om^H$. Our aim in this step is to show a lower
  bound for the pressure.

  On the lower boundary $\Sigma$ there holds $p_H = p_0 \ge 0$.  We
  claim that there is a lower bound also along the upper boundary
  $\Sigma^h$ of $\Omega^h$. Indeed, by definition of $h$ in \eqref
  {eq:DefhH}, there is a subset of non-vanishing measure in
  $(0,L) \times (h-1,h)$ on which $\del_z s_H > 0$ holds, i.e.
  $p_H > p_c(s_H) \ge 0$. The Lipschitz bound \eqref{eq:FingrDP}
  implies that $p_H \ge -C_L$ holds on $\Sigma^h$ for
  $C_L = C_P \sqrt{1 + L^2}$.
      
  We can now exploit a maximum principle to obtain
  \begin{equation}
    -C_L \leq p_H \leq \bar{p} \quad \text{ a.e. in } \Om^h\,.
    \label{eq:PHisBounded}
  \end{equation}
  The maximum principle is derived by using $[p_H + C_L]_-$ as a
  test function in \eqref{eq:TWRichards-TTW}, which results in
  $$\int_{\Om^h}[p_H + C_L]_-\nabla\cdot[k(s_H)\nabla
  p_H]=\int_{\Om^h} [p_H + C_L]_-(c-g k'(s_H))\del_z s_H\,.$$ An
  integration by parts yields
  \begin{align*}
    &\int_{\Om^h}k(s_H)|\nabla [p_H + C_L]_-|^2 =
      \int_{\Sigma^h} [p_H + C_L]_- k(s_H) \del_z p_H\\
    &\qquad -\int_{\Sigma} [p_H + C_L]_- k(s_H) \del_z p_H  
      +  \int_{\Om^h} [p_H + C_L]_- (c-g k'(s_H)) \del_z s_H\,.
  \end{align*}
  As analyzed before, the boundary terms vanish because of
  $p_H + C_L \ge 0$ along $\Sigma$ and along $\Sigma^h$.  Regarding
  the last integral we note that in every point $x$ with
  $\del_z s(x) >0$, there holds
  $p_H(x) \ge p_c(s_H(x))\geq p_c(s_*) = 0$, and hence
  $[p_H + C_L]_- = 0$. This shows that all terms on the right hand
  side vanish.  We obtain \eqref{eq:PHisBounded}.
 
  \smallskip {\em Step 2: A further a priori estimate.}  From the
  uniform pressure bound \eqref{eq:PHisBounded} we conclude that
  $p_c^{-1}(p_H)$ is bounded away from $1$. With this information, the
  bound of \eqref{Eq:MainAprioriEstimate1} provides, with a constant
  $C>0$ independent of $H$, the inequality
  $$ \int_{\Om^h} \cut_0\, |\nabla p_H|^2 + \int_{\Om^H} \cut_{\!_>}\,
  {p_c}'(s_H)  |\nabla s_H|^2 \leq C\,.
  $$ Similarly, \eqref{Eq:MainAprioriEstimate2} implies
  $$
  \int_{\Om^H} \cut_{\!_>}|\nabla p_H|^2 + \int_{\Om^H} |\nabla
  (\del_z s_H)|^2 \leq C.
  $$ Combining both of these inequalities with \eqref{eq:PHisBounded},
  and recalling $\del_z s_H = 0$ in $\Omega^H\setminus \Om^h$, we
  obtain
  \begin{equation}
    \label{eq:bdd-p-5123}
    \max_{\Om^h} |p_H|^2 + \int_{\Om^h} |\nabla p_H|^2 + \int_{\Om^H}
    [|\del_z s_H|^2 + |\nabla \del_z s_H|^2] \leq C\,.
  \end{equation}
  
  \smallskip {\em Step 3: Limit $H\to \infty$.}  Because of
  $h\to \infty$, we find a limiting pair $(s,p)$ such that the local
  convergences of \eqref{eq:Convergence_sHpH} hold for any compact
  subset $\Om'$ of $\Om$. It is straightforward to verify that $(s,p)$
  solves \eqref{eq:GDETW}.  Moreover, \eqref {eq:bdd-p-5123} together
  with $h\to \infty$ implies the additional properties
  $\nabla p\in L^2(\Om)$ (as a bounded solution to an elliptic
  equation) and $\del_z s\in L^2(\Om)$.
  
  \smallskip Regarding the limiting flux, we start from the relation
  $\int_{\Sigma^H} k(s_H)(\p_z p_H +g) = F_\infty$. In order to
  calculate limits, we once more use the quantity $F_c^H(z)$ of \eqref
  {eq:flux-H-z}, which is independent of $z$.  The local strong
  convergence of $s_H$ and the local weak convergence of $\nabla p_H$
  yield, for almost every $z$, as $H\to \infty$,
  \begin{align*}
    &\int_0^L k(s(y,z)) [\del_z p(y,z) + g] - c s(y,z)\, dy
      =:  F_c(z) \\
    &\quad \leftarrow F_c^H(z)
      := \int_0^L k(s_H(y,z)) [\del_z p_H(y,z) + g] - c s_H(y,z)\, dy\\
    &\quad = F_c^H(H) = \int_0^L k(s_H(y,H)) [\del_z p_H(y,H) + g]
      - c s_H(y,H)\, dy\\
    &\quad = F_\infty - c \int_0^L s_H(y,H) \, dy\,.
  \end{align*}
  Because of $s_H(y,H) \ge s_H(y,z)$ for every $z$, and $s_H \to s$,
  there holds
  $$\lim_{z\to\infty} \int_0^L s(y,z)\, dy \le \lim_{H\to\infty}
  \int_0^L s_H(y,H) \, dy\,.$$ Taking in the above calculation both
  limits, $z\to \infty$ and $H\to\infty$, exploiting
  $\nabla p\in L^2(\Om)$, we find
  $$\int_{0}^L g k(s^*(y))\, dy = 
  \lim_{z\to \infty} \int_{0}^L g k(s(y,z))\, dy \le F_\infty \,.
  $$
  This concludes the proof.
\end{proof}

\begin{remark}[Both solution types occur]
  The one-dimensional travelling wave results in
  \cite{el2018traveling} indicate that both solution types exists for
  a given $s_*\in (0,1)$ and $F_\infty$ satisfying \eqref{eq:FingrCF}.
  Type I (large) solutions occur in the one-dimensional model when
  $\t$ is large. On the other hand, if
  $\| p_0 - p_c(s_*) \|_{L^\infty(\Sigma)}$ is small, then Type II
  (small) solutions are expected to occur for small $\t$ values.  Our
  numerical results confirm that both solution types occur.
\end{remark}

We finally want to show that, for a given flux $F_\infty$, it is
possible to find a wave-speed $c$ such that condition
\eqref{eq:ConditionAtSigma} is satisfied.

\begin{theorem}[Selecting a wave-speed $c$ in dependence of $F_\infty$
  and $s_*$]\label{thm:c-F-infty}
  Let $\t>0$, $s^*\in (0,1)$, and boundary data
  $s_0,\,p_0\in C^1(\Sigma)$ be given, $p_c(s_*)\leq p_c(s_0) < p_0$
  on $\Sigma$, furthermore $F_\infty$ in the bounds of
  \eqref{eq:FingrCF}.  We assume that, for all $c\in [c_1, c_2]$ with
  $c_1 := k'(s_*)g$ and $c_2 := g (k(1)-k(s_*))/(1-s_*)$, a sequence
  $(s_H,p_H)$ of solutions to \eqref{eq:TTW} satisfying Assumption
  \ref {ass:unbounded-dom} exists. We consider the corresponding limit
  solutions $(s,p)$ and their fluxes
  \begin{equation}\label{eq:WavespeedSelection}
    F_c = \int_{\Sigma} (k(s_0)[\p_z p + g] - c s_0)\,,
  \end{equation}
  and assume that $F_c$ depends continuously on $c$.  Then there
  exists a wave-speed $\bar{c} \in (c_1, c_2)$ such that the
  corresponding pair $(s,p)$ satisfies \eqref {eq:ConditionAtSigma},
  $F_c = (g k(s_*)-cs_*) L$.
\end{theorem}

\begin{proof}
  We consider the continuous function $G : [c_1, c_2]\to \R$ 
  \begin{equation}
    G(c) := F_c - (g k(s_*)-cs_*) L\,. 
  \end{equation}
  We recall that $G$ depends in an explicit way on $c$, but also
  implicitely, since $s$ and $p$ (and hence $F_c$) depend on $c$.  The
  flux quantity $F_c$ is independent of $z$, we choose to evaluate it
  at $z\to \infty$. We denote the limit of the first two terms as
  $$F_0 := F_0(c) := \lim\limits_{z\to \infty}
  \int_0^L k(s(y,z))[\p_z p(y,z) + g]\,dy\,.$$
  We observe that, by Theorem \ref {theo:main},
  $$
  F_0 =
  \begin{cases}
    F_\infty &\text{ if }\quad g\int_0^L k(s^*)> F_\infty\,,\\
    g\int_0^L k(s^*) &\text{ if }\quad g\int_0^L k(s^*)\le F_\infty\,.
  \end{cases}
  $$
  In both cases holds $F_0 \le F_\infty$ and
  $F_0 \le g\int_0^L k(s^*)$.  The function $G$ can be written as
  $$
  G(c) = F_0 - k(s_*)gL - c\int_0^L (s^*(y)-s_*)\,dy\,.
  $$
  \smallskip {\em Showing $G(c) > 0$ as $c\to c_1$.}  If the
  solution is of Type II (small solution, second case in the above
  distinction), then
  $$
  G(c) = \int_0^L (s^*(y)-s_*)\,dy\
  \left (g\ \frac{\int_0^L (k(s^*(y))-k(s_*))\,dy}
    {\int_0^L (s^*(y)-s_*)\,dy}\ - c\right )\,.
  $$
  We exploit that $s^* > s_0 \ge s_*$ implies, for every $y\in (0,L)$,
  that $k(s^*(y))-k(s_*) > k'(s_*) (s^*(y)-s_*)$. This implies that,
  for $c$ close to $c_1 = k'(s_*)g$, there holds $G(c)>0$. On the
  other hand, if $(s,p)$ is of Type I (large solutions), then
  \begin{align*}
    G(c) &= F_\infty - k(s_*)gL - c\int_0^L (s^*(y)-s_*)\,dy\\
         &\geq F_\infty -  k(s_*)gL - c(1-s_*)L\\
         &\ge g L  k'(s_*)(1-s_*) - c(1-s_*)L + \eps\\
         &= (g k'(s_*) - c) (1-s_*) L + \eps\,,
  \end{align*}
  where we exploited the lower bound
  $g L [k(s_*) + k'(s_*)(1-s_*)] + \eps \le F_\infty$ for some
  $\eps>0$.  We see that, also in this case, for $c$ close to
  $c_1 = k'(s_*)g$, there holds $G(c)>0$.
  
  \smallskip {\em Showing $G(c) < 0$ as $c\to c_2$.}  Consider
  solutions of Type II (small solutions). For
  $\mu:=(k(1)-k(s_*))/(1-s_*)$, we show that in this case, there
  exists $\e>0$ independent of $c\in [c_1,c_2]$ such that
  \begin{equation}
    \int_0^L  (k(s^*)-k(s_*)) \leq \left (\mu  - \e\right) \int_0^L 
    (s^*-s_*)\,.\label{eq:KrLessthanKr1}
  \end{equation}
  Since $(k(s)-k(s_*))/(s-s_*)$ is a strictly increasing function for
  $s>s_*$, $\int_0^L (k(s^*)-k(s_*))= \mu \int_0^L (s^*-s_*)$ if and
  only if $s^*(y)\in \{s_*,1\}$ for all $y\in (0,L)$. From Jensen's
  inequality, one has
  \begin{align}
    k\left (\frac{1}{L}\int_0^L s^* \right )
    &\leq  \frac{1}{L}\int_0^L  k(s^*)
    \leq \frac{F_\infty}{gL}< k(1)\,,\label{eq:IntLessThan1}
  \end{align}
  implying $\frac{1}{L}\int_0^L s^* <1$. Hence, the possibility
  $s*\equiv 1$ in $(0,L)$ is ruled out. Moreover, since
  $s^*>s_0\geq s_*$, the possibility $s*\equiv s_*$ in $(0,L)$ is also
  ruled out. From \Cref{lem:NablaSH}, $\|\nabla s\|_{L^\infty(\Om)}$
  is bounded. Hence $s^*$ cannot take both the values $s_*$ and $1$
  without transitioning through the intermediate values. Thus
  \eqref{eq:KrLessthanKr1} holds.
  
  If the solution is of Type II, then, for $c$ close enough to
  $g\mu=g(k(1)-k(s_*))/(1-s_*)$, we obtain from
  \eqref{eq:KrLessthanKr1},
  \begin{align*}
    G(c)&= F_0 - k(s_*)gL - c\int_0^L (s^*-s_*)
          =  \int_0^L [g(k(s^*)-k(s_*)) - c (s^*-s_*)]\\
        &\leq \int_0^L (s^*-s_*) \left [g\mu -g\e - c\right ]\leq 0\,.
  \end{align*}
  If the solution is of Type I, then for $g_F> 0$ as defined in
  \eqref{eq:gFdef} (see also \Cref{prop:SolCase1exists}), one has
  \begin{align*}
    G(c)&= F_\infty - k(s_*)gL - c\int_0^L (s^*-s_*)\\
        &=  \int_0^L [g(k(s^*)-k(s_*))- c (s^*-s_*)] - g_F \int_0^L k(s^*)\\
        &\leq (g\mu-c)\int_0^L (s^*-s_*) -g_F \int_0^L k(s_0)\,.
  \end{align*}
  Consequently, $G(c)<0$ for $c$ close enough to $c_2 =g\mu$.  Hence,
  there exists a zero $\bar{c}$ of $G(\cdot)$ in $(c_1,c_2)$. This was
  the claim.
\end{proof}

\section{Numerics}
\label{sec.numerics}

\subsection{Numerical solution of system \eqref{eq:TTW}}
\label{ssec.num-p1}

The primary numerical task is to solve system \eqref{eq:TTW} for $s$
and $p$, where the speed $c$ and the total influx $F_\infty$ are
given. The existence of a solution was established in Theorem \ref
{theo:Wellposedness-TTW}. We use an iterative method in order to deal
with the nonlinearities. With a positive number $M>0$, we use the
iteration $(s^{i-1}, p^{i-1}) \mapsto (s^i, p^i)$ that is given by
\begin{subequations}\label{eq:numItScheme}
  \begin{align}
    \label{eq:numItScheme1}
    M p^i - \nabla \cdot[k(s^{i-1})(\nabla p^i + ge_z)]
    &= Mp^{i-1} - \frac1{\tau}[p^{i-1} - p_c(s^{i-1})]_+\,,\\
    \label{eq:numItScheme2}
    \partial_z s^i- \eps \Delta s^i
    &= \frac1{c\tau}[p^{i} - p_c(s^{i-1})]_+\,.
  \end{align}
\end{subequations}
The equations are solved in the rectangular computational domain
$\Omega^H$ for some initial guess $(s^0,\, p^0)$. They are
supplemented by the boundary conditions \eqref
{eq:LowerBoundary-TTW}--\eqref {eq:Flux-TTW} and no-flux conditions at
the lateral boundaries.

For $\eps=0$, a fixed point of the iteration scheme
\eqref{eq:numItScheme} provides a solution of \eqref{eq:TTW}. The
set-up is such that the equations can be solved subsequently: One can
solve the first equation for $p^i$, then the second equation for
$s^i$. The iteration strategy is based on the L-scheme
\cite{MITRA2018}, the iteration is expected to converge for
$M\geq \t^{-1}$, irrespective of the initial guess. We introduce an
elliptic regularization in the second equation (which is first order
in $s$), numerical experiments are run with a small number $\eps>0$.

In order to discretize \eqref{eq:numItScheme}, we introduce a uniform
triangulation $\Om^H_{\mathrm{h}}$ of the domain $\Omega^H$ and apply
linear finite elements. In this sense, the discretization is based on
the weak formulation in \eqref {eq:W1q-s-H}--\eqref
{eq:WeakSolSplus-TTW}. The resulting scheme has been implemented in
the adaptive finite element tool box AMDiS \cite{amdis}. The linear
equations arising from the discretization are treated with the direct
solver UMFPACK, \cite{umfpack}.

The physical parameters of the problem are chosen as in \cite
{LamaczRaetzSchweizer2011},
\begin{align}
  \label{eq:num:functions}
  p_c(s) = s\,, \quad 
  k(s)=
  \begin{cases}
    \kappa & \quad \text{for} \quad s < a\,,\\
    \kappa + (s - a)^2 & \quad \text{for} \quad s \ge a\,,
  \end{cases}
\end{align}
and 
\begin{equation*}
  \label{eq:num:parameters}
  g = 1\,, \quad \tau = 2\,, \quad \kappa = 0.001\,, \quad a = 0.32\,, \quad
  F_\infty = 0.056\,.
\end{equation*}
The domain is $\Omega^H = (-1,1)^2$; up to a shift of the domain, this
coincides with $L=2$ and $H=2$ in analytical results.  The parameters
for the numerical code are
\begin{equation*}
  M = 4\,,\quad \eps = 0.0008\,.
\end{equation*}
The initial values for the iteration have been chosen as
\begin{equation*}
  p^0 = 4.5\,, \quad s^0 = 10^{-5}\,.  
\end{equation*}
Regarding the lower boundary, we use the constant function
$s_0 = 10^{-5}$ and the slightly perturbed pressure boundary condition
\begin{equation*}
  p_0(y) = p_c(s_0) + \delta e^{-(y/d)^2}\,.
\end{equation*}
The postive parameter $\delta = 0.078$ measures the amplitude of the
perturbation and the scaling factor $d=0.25$ measures the width of the
perturbation.

Figure \ref {fig:num:varyCConstSatInit} shows results for four
different values of the speed $c$. We see a remarkable difference
between the solution for $c=0.04785$ and the solution for
$c=0.04786$. The abrupt change finds its counterpart in Theorem \ref
{theo:main} (we recall that the theorem is treating unbounded domains
while the numerical results are for a fixed bounded domain): The two
images on the left show Type I solutions, i.e., ``large solutions''
with a free boundary. The two images on the right show ``small
solutions''.
\begin{figure}[H]
  \centering
  \includegraphics[width=0.21\textwidth]{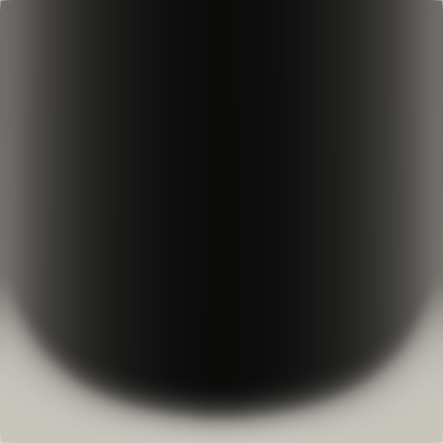}
  \includegraphics[width=0.21\textwidth]{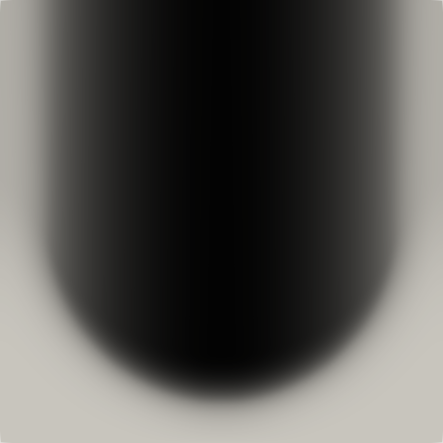}
  \includegraphics[width=0.21\textwidth]{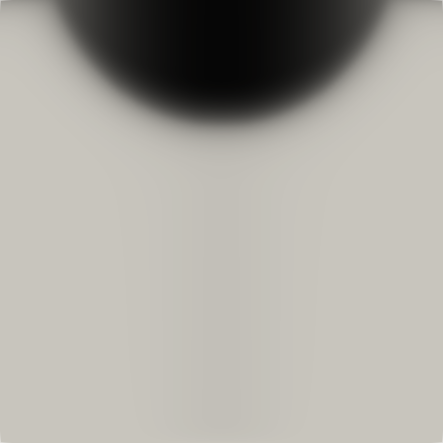}
  \includegraphics[width=0.21\textwidth]{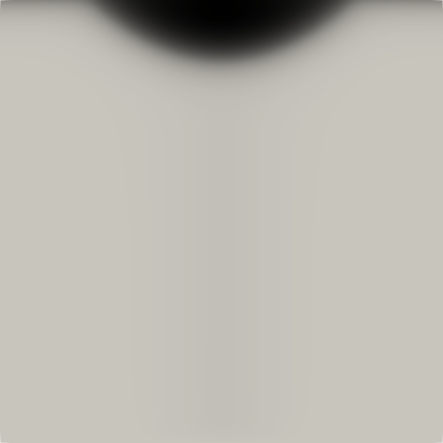}
  \includegraphics[width=0.09\textwidth]{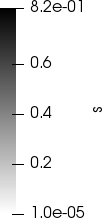}
  \caption{The discrete solutions $s_{\!_\mathrm{h}}$ of the iteration
    scheme for (from left to right) $c=0.04$, $c=0.04785$,
    $c=0.04786$, $c=0.05013166020$.}
  \label{fig:num:varyCConstSatInit}
\end{figure}

In the above experiments, we have solved system \eqref {eq:TTW} for
different values of $c$. We now ask: What is the correct wave speed
$c$ in the sense of \eqref {eq:ConditionAtSigma}? We use the following
finite domain approximations: $s_*$ can be neglected, hence, in
particular, $k(s_*) = \kappa$. Furthermore, $s_0$ is constant and so
small that also $k(s_0)$ can be replaced by $\kappa$. Condition \eqref
{eq:ConditionAtSigma} then reads
$G_1(c) := c - (\kappa \int_\Sigma \del_z p_{\!_\mathrm{h}})/(L s_0)
\stackrel{!}{=} 0$. We find the values as displayed in Table
\ref{table1}. 
\begin{table}[ht]
  \centering
  \begin{tabular}[t]{| c | c | c |}
    \hline
    $c$ & $0.0476$ & $0.0477$ \\
    \hline
    $G_1(c)$ & $0.0476 + 0.0218 > 0$ & $0.0477 - 0.3304 < 0$ \\
    \hline
  \end{tabular}
  \caption{\label{table1} Values for $G_1(c)$ for various $c$-values.}
\end{table}
%
We conclude that $G_1(\bar{c}_1) = 0$ is satisfied for some
$\bar{c}_1 \in [0.0476,0.0477]$. Up to the above finite domain
approximations, we expect the travelling wave speed to be about
$0.0477$.  This is remarkably close to the jump point, compare Figure
\ref {fig:num:varyCConstSatInit}. We furthermore note that the value
is not far from the value $c=0.053$ that can be extracted from
simulation results reported in \cite{LamaczRaetzSchweizer2011}.

\subsection{Path-following algorithm to adjust $c$}
\label {ssec.num-p2}

So far, for each value of $c$, we started the iterative scheme
\eqref{eq:numItScheme} with constant functions $s^0$ and $p^0$ as
initial guess.  Since we are interested in solutions for a whole range
of $c$-values, there is a very natural idea to speed up calculations:
After having changed the value of $c$, instead of starting the
iterative scheme from scratch, we start the iteration with the
solution of the last value of $c$. Thereby, we increased $c$ in every
interation step by $10^{-4}$ in some experiments, by $10^{-11}$ in
others. 

Interestingly, it turns out that this scheme produces results that are
different from those reported in \Cref {ssec.num-p1}.
Results are displayed in Figure \ref{fig:num:varyC}, and once more, we
observe that, below a critical value for $c$, solutions are ``large
solutions'', above the critical value, we find ``small
solutions''. This feature is as in the sequence of Figure \ref
{fig:num:varyCConstSatInit}, but the critical value of $c$ is now
different: It is about $\bar{c}_2 = 0.050$ and no longer about $\bar{c}_1 =
0.048$. For values of $c$ below $\bar{c}_1$ and for values above $\bar{c}_2$, the
results of the two schemes coincide.

\begin{figure}[H]
  \centering
  \includegraphics[width=0.21\textwidth]{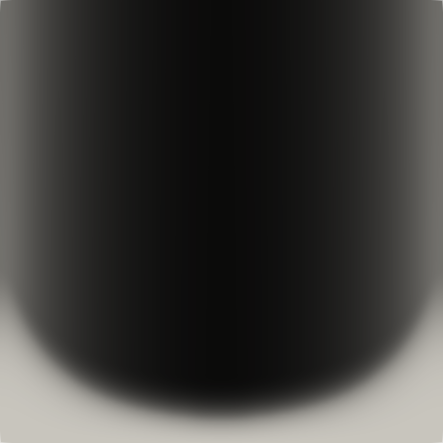}
  \includegraphics[width=0.21\textwidth]{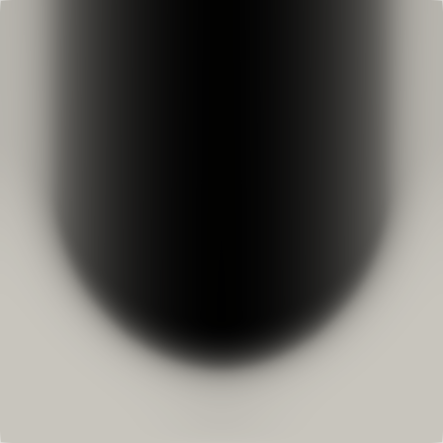}
  \includegraphics[width=0.21\textwidth]{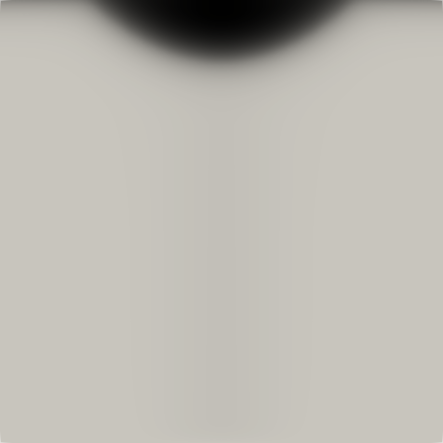}
  \includegraphics[width=0.21\textwidth]{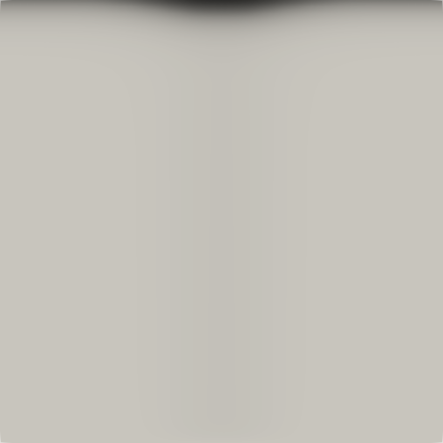}
  \includegraphics[width=0.09\textwidth]{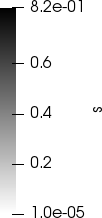}
  \caption{Plots of the discrete solutions $s_{\!_\mathrm{h}}$ of the
    path-following iteration scheme for (from left to right) $c=0.04$,
    $c=0.05013166020$, $c=0.05013166023$, $c=0.0625$.}
  \label{fig:num:varyC}
\end{figure}

We conclude with an evaluation of the integral condition in Theorem
\ref{theo:main}, where the criterion for a ``large solution'' was
$g \int_0^L k(s^*(y))\, dy \geq F_\infty$ for the saturation values
$s^*(y) := \lim_{z\to \infty} s(y,z)$ at infinity.  With the
approximation $s^* \approx s|_{\Sigma^H}$ and with
\eqref{eq:Flux-TTW}, the criterion for a ``large solution'' reads
\begin{equation}
  \label{eq:intForSolTypes}
  G_2(c) := \int_{\Sigma^H} k(s) \partial_z p \leq  0\,.
\end{equation}
Our simulations yield the values in Table \ref{table2}. We observe
that the change of sign of $G_2$ occurs only after the point that the
solution switched to the ``small solution''.
\begin{table}[ht]
  \centering
  \begin{tabular}[t]{| c | c | c | c | c |}
    \hline
    $c$ & $0.04$ & $0.05013166020$ & $0.05013166023$ & $0.0625$ \\
    \hline
    $G_2(c)$ & $-0.1948$ & $-0.1630$
                                   & $-0.1492$ & $0.0183$ \\
    \hline
  \end{tabular}
  \caption{\label{table2} Values for $G_2(c)$ for various $c$-values.}
\end{table}

Our observations may be interpreted as follows: For a range of values
of $c$, there are two solutions of system \eqref{eq:TTW}. This is not
in contradiction with our analysis, since Theorem \ref
{theo:Wellposedness-TTW} provides the existence, but not the
uniqueness of solutions. A numerical scheme has the tendency to find
the ``stable'' solution (``stable'' has to be interpreted
appropriately). In a path-following code as described here (in
\Cref{ssec.num-p2}), due to numerical stabilization
aspects, the code can follow one path beyond the point where it looses
stability. We conjecture that this is what is visible in the
observation $\bar{c}_2 > \bar{c}_1$.

\section*{Conclusions}

We studied the travelling wave equations for a porous media imbibition
problem with hysteresis. Denoting by $c$ the unknown speed of the
travelling wave, we treat a free boundary problem with an additional
parameter.  Our analysis shows that, after a domain truncation and for
boundary conditions within physically reasonable limits: (i) For a
prescribed speed $c$, travelling wave solutions exist. In the limit of
infinite domains, different types of limit solutions can occur.  (ii)
A critical wave speed $c$ can be selected by a flux condition. (iii)
Numerical experiments provide solutions with the shape of a finger.
We find values of $c$ that are in good agreement with time-dependent
calculations.  Different numerical algorithms yield slightly different
values for $c$, an effect that may be related to non-uniqueness of
solutions.

\appendix 
\section{Appendix}

The following result on solution sequences $(s_H,p_H)$ does not rely
on Assumption \ref {ass:unbounded-dom}, but follows directly from the
variational principle.

\begin{lemma}[Large solution sequences have unbounded pressure]
  \label{lem:unbounded-p-criterion}
  For a sequence $0 < H\to \infty$, let $(s_H,p_H)$ be solutions to
  \eqref{eq:TTW}. We assume that, for some height parameter $z_0>0$
  and some bound $C_k>0$, every solution $s_H$ satisfies the integral
  condition
  \begin{equation}
    \label{eq:gk-large}
    \int_0^L g k(s_H(y,z_0))\, dy \ge C_k > F_\infty\,.
  \end{equation}
  In this situation, the sequence of pressure functions is unbounded,
  \begin{equation}
    \label{eq:press-unbounded}
    \| p_H \|_{L^\infty} \to \infty\,.
  \end{equation}
  In particular, it generates a ``large'' Type I solution.
\end{lemma}

\begin{proof}
  For a contradiction argument we assume that, for some $\bar p>0$,
  the pressure functions are bounded, $|p_H| \le \bar p$ on
  $\Om^H$. We recall that $p_H$ is the minimizer for the functional
  $A$ of \eqref {eq:A-var-principle}, for given $s = s_H$. This
  provides a lower bound for $A$: For any function $\fhi\in X_{p_0}$,
  there holds, by Lemma \ref {lem:Jensen}, 
  \begin{align*}
    A(\fhi) \ge A(p_H)
    &\ge \int_{\Om^H} \frac12 k(s_H) |\nabla p_H + g e_z|^2
      - F_\infty L \bar p\\
    &\ge \frac12 g^2 \left(\int_{\Om^H} k(s_H)\right)
      - C_1(\bar p) - F_\infty L \bar p\,.
  \end{align*}

  Our aim is to find a contradiction, which we obtain by constructing
  a comparison function with lower energy. We choose a function
  $\tilde p_H$ that connects, in the domain $\{ z\in (0,1)\}$, the
  boundary data $p_0$ in a smooth way with $\tilde p_H \equiv 0$ for
  $z=1$. For larger $z$, we set $\tilde p_H(y,z) = -g_F (z-1)$, where
  the coefficient $g_F\in (0,g)$ is chosen below. We calculate for the
  energy
  \begin{align*}
    A(\tilde p_H) \le C_2
    + \frac{1}{2} | g - g_F|^2 \left(\int_{\Om^H} k(s_H)\right)
    + F_\infty H g_F\,.
  \end{align*}

  Combining the two inequalities and using
  $\bar C_k := \left( g \int_0^L \int_{z_0}^H k(s_H)\right)/(H-z_0)$,
  we find
  \begin{equation}\label{eq:834z}
    \frac12 g \bar C_k H \le C_3 + F_\infty H g_F
    + \frac{H \bar C_k}{2g} | g - g_F|^2 \,.
  \end{equation}
  Optimizing in $g_F$ leads to the choice $g_F := g - q$ with
  $q := (gF_\infty)/\bar C_k < g$. In order to compare the prefactors of
  $H$ on both sides we study
  \begin{align*}
    &\frac12 \bar C_k g - \frac{\bar C_k}{2g} | g - g_F|^2 - F_\infty  g_F
    =\frac12 \bar C_k g - \frac{\bar C_k}{2g} | q |^2 - F_\infty  (g-q)\\
    &\qquad = \frac{\bar C_k}{2g}\left(g^2 - q^2 - 2 q (g-q)\right)
      = \frac{\bar C_k}{2g}\left(g^2 +  q^2 - 2 q g \right)
      = \frac{\bar C_k}{2g} (g-q)^2 > 0\,.
  \end{align*}
  For large $H$, this yields a contradiction in \eqref {eq:834z}.
\end{proof}

\begin{lemma}[A Jensen type inequality]
  \label{lem:Jensen}
  For $\Omega^H = (0,L)\times (0,H)$ with points $x = (y,z)$,
  $k:\Omega^H\to [0,k_0]$ monotonically increasing in $z$, and
  $u : \Omega^H\to \R$ with the uniform bound
  $\| u \|_{L^\infty} \le \bar u$, there exists a constant
  $C_1 = C_1(\bar u, k_0)$, independent of $H$, such that
  \begin{equation}
    \label{eq:Jensen}
    \int_{\Omega^H} k\, |\nabla u + g e_z|^2 \ge -C_1 + g^2 \int_{\Omega^H} k\,.
  \end{equation}
\end{lemma}

\begin{proof}
  We use the averaging operator $M : L^2(\Omega^H) \to \R$, defined
  by
  $$M(v) := \left( \int_{\Omega^H} k v\right) \big/ \left(
    \int_{\Omega^H} k \right)\,.$$ This operator is linear and maps
  the constant function $v\equiv a \in \R$ to $M(v) = a$. We
  furthermore use the convex function $\psi : \R\to \R$,
  $\xi \mapsto | \xi + g |^2$.  Jensen's inequality provides
  $$M(\psi(\del_z u)) \ge \psi(M(\del_z u))\,.$$
  In our setting and with $m := \int_{\Omega^H} k$, this yields
  \begin{align*}
    \int_{\Omega^H} k\, |\nabla u + g e_z|^2
    &\ge \int_{\Omega^H} k\, |\del_z u + g|^2
      = m\, M(\psi(\del_z u))
      \ge m\, \psi(M(\del_z u))\,.
  \end{align*}
  We calculate, using that $k$ is increasing in $z$,
  \begin{align*}
    |M(\del_z u)| &= \left| \frac1{m} \int_{\Omega^H} k \del_z u \right|
                    \le \frac1{m} \left| \left. \int_0^L k\, u \right|_0^H
                    - \int_{\Omega^H} \del_z k\, u \right|\\
    &\le \frac1{m} \left( 2 k_0 \bar u + \bar u \int_{\Omega^H} \del_z k \right)
    \le \frac{3k_0 \bar u}{m}\,.
  \end{align*}
  Inserting above we obtain
  \begin{align*}
    \int_{\Omega^H} k\, |\nabla u + g e_z|^2
      \ge m\, \psi(M(\del_z u))
    = m\, | g + M(\del_z u))|^2
    \ge m g^2 - 6 g k_0 \bar u\,.
  \end{align*}
  This shows the claim.
\end{proof}

\begin{lemma}[Lipschitz continuity of $s_H$]
  \label{lem:NablaSH}
  Let $F_\infty, c, \t, s_* > 0$ and $s_0,p_0\in C^1(\Sigma)$ with
  $s_*\leq s_0<1$ and $p_0 \ge p_c(s_0)$ be fixed. For $H>0$, let
  $(s_H,p_H)$ be the $TW_H$-solution to \eqref{eq:TTW} satisfying
  \eqref{eq:FingrDP} and \eqref {eq:ass-s-regularity}. Then, for
  $\rho=\min\{{p_c}'\}>0$, there holds
  \begin{equation}
    \label{eq:s-Lip-claim}
    \|\del_z s_H\|_{L^\infty(\Om^H)}, \|\del_y s_H\|_{L^\infty(\Om^H)}\leq C_P/\rho +
    \| \del_y s_0 \|_{L^\infty(\Sigma)} +
    \tfrac{1}{c\t}\| p_0-p_c(s_0) \|_{L^\infty(\Sigma)} =: C_s\,.
  \end{equation}
\end{lemma}

\begin{proof}
  To prove the lemma, we consider a regularization of the signum
  function, denoted as $\sgn_\e:\R\to [-1,1]$. A possible choice is
  $\sgn_\e(\eta) := \eta/\e$ for $\eta\in [-\e,\e]$,
  $\sgn_\e(\eta):=-1$ for $\eta<-\e$ and $\sgn_\e(\eta):=1$ for
  $\eta>\e$. We also introduce the primitive
  $H_\e(\eta)=\si_0^\eta \sgn_\e(\vr)d\vr$. We demand that, as
  $\e\to 0$, there holds $\sgn_\e(\eta)\to \sgn(\eta)$,
  $H_\e(\eta)\to |\eta|$, and $\eta\, \sgn_\e(\eta)\to |\eta|$.

  We differentiate relation \eqref {eq:HysDyn-TTW}, in the sense of
  distributions, with respect to $x_j$ for $x_j = y$ and for
  $x_j = z$.  The regularity assumption \eqref {eq:ass-s-regularity}
  on $s_H$ allows to write
  $$
  c\t \p_z \p_{x_j} s_H + {p_c}'(s_H) \p_{x_j} s_H\ \sgn(\p_z s_H)
  = \p_{x_j} p_H\  \sgn(\p_z s_H)\, .
  $$
  Multiplying both sides with $\sgn_\e(\p_{x_j} s_H)$ yields
  \begin{equation}
    \begin{split}
      &c\tau \p_z H_\e(\p_{x_j} s_H)+ {p_c}'(s_H)\,
      \sgn_\e(\p_{x_j} s_H) \p_{x_j} s_H\, \sgn(\p_z s_H) \\
      &\qquad = \sgn(\p_z s_H)\, \sgn_\e(\p_{x_j}s_H) \p_{x_j} p_H.
    \end{split}\label{eq:sign-eps-testing}
  \end{equation}
  Passing to the limit $\e\to 0$, we obtain for $x_j = z$ the relation
  \begin{equation}\label{eq:s-Lip-24}
    c\t\p_z |\del_z s_H| + {p_c}'(s_H)|\del_z s_H|
    \leq |\del_z p_H| \leq C_P\,,
  \end{equation}
  where we used \eqref {eq:FingrDP} in the last inequality.  We
  exploit that, for $z=0$, there holds
  $\p_z s_H= \tfrac{1}{c\t}[p_0-p_c(s_0)]_+$, and hence also
  $|\p_z s_H| \le \tfrac{1}{c\t}\| p_0-p_c(s_0)
  \|_{L^\infty(\Sigma)}$.  Inequality \eqref {eq:s-Lip-24} implies
  that $|\del_z s_H|$ cannot exceed the value $C_s$ of \eqref
  {eq:s-Lip-claim}.

  We now study $x_j = y$ in \eqref {eq:sign-eps-testing}. In the limit
  $\e\to 0$, exploiting $\del_z s_H \ge 0$, we find
  \begin{equation}\label{eq:s-Lip-25}
    c\t\p_z |\del_y s_H| + {p_c}'(s_H)|\del_y s_H|\, \sgn(\p_z s_H)
    \leq |\del_y p_H|\,  \sgn(\p_z s_H)\,.
  \end{equation}
  With the uniform bound $|\nabla p_H| \le C_P$ of \eqref {eq:FingrDP}
  we can write
  \begin{equation}\label{eq:s-Lip-82}
    c\t\p_z |\del_y s_H| \le (C_P - \rho |\del_y s_H|)\, \sgn(\p_z s_H)\,.
  \end{equation}
  For $z=0$, there holds $\p_y s_H = \del_y s_0$.  Inequality \eqref
  {eq:s-Lip-82} implies that $|\del_y s_H|$ cannot exceed the value
  $C_s$ of \eqref {eq:s-Lip-claim}.
\end{proof}

\bibliographystyle{plain}

\def\cprime{$'$}

\end{document}